\documentclass{amsart}
\usepackage{amsxtra}
\usepackage{amssymb,amsmath,amsthm,stmaryrd}
\usepackage{color}
\usepackage[all]{xy}
\usepackage{hyperref}
\usepackage{appendix}
\usepackage{cancel}
\theoremstyle{plain}
 
\newtheorem{mythm}{Theorem}[]
\newtheorem*{mydef}{Definition}
\newtheorem{theorem}{Theorem}[section]
\newtheorem{lemma}[theorem]{Lemma}
\newtheorem{proposition}[theorem]{Proposition}
\newtheorem{corollary}[theorem]{Corollary}
\newtheorem{definition}[theorem]{Definition} \theoremstyle{definition}
\newtheorem{example}[theorem]{Example}
\newtheorem{remark}[theorem]{Remark}
\newtheorem*{thm}{Theorem}

\newcommand{\lie}[1]{\mathfrak{#1}}

\newcommand{\R}{\mathbb{R}} 
\newcommand{\Z}{\mathbb{Z}}
\newcommand{\inv}{^{-1}}
\newcommand{\N}{\mathbb{N}}
\newcommand{\mx}{\mathfrak{X}}
\newcommand{\dr}{\mathbf{d}}
\newcommand{\ldr}[1]{{{\pounds}}_{#1}}

\newcommand{\E}{\underline{E}}
\DeclareMathOperator{\rank}{rank}

\DeclareMathOperator{\Hom}{Hom}
\DeclareMathOperator{\End}{End}
\DeclareMathOperator{\Id}{Id}

\DeclareMathOperator{\Pont}{Pont}
\DeclareMathOperator{\tr}{tr}
\DeclareMathOperator{\gtr}{gtr}
\DeclareMathOperator{\ad}{ad}

\begin{document}

\title{Obstructions to representations up to homotopy and ideals}


\author{M.~Jotz Lean}
\address{Mathematisches Institut, Georg-August Universit\"at G\"ottingen.}
\email{madeleine.jotz-lean@mathematik.uni-goettingen.de}

\begin{abstract}
  This paper considers the Pontryagin characters of graded vector
  bundles of finite rank, in the cohomology vector spaces of a Lie
  algebroid over the same base.  These Pontryagin characters vanish if
  the graded vector bundle carries a representation up to homotopy of
  the Lie algebroid. As a consequence, this gives a strong obstruction
  to the existence of a representation up to homotopy on a graded
  vector bundle of finite rank. In particular, if a graded vector
  bundle $E[0]\oplus F[1]\to M$ carries a $2$-term representation up
  to homotopy of a Lie algebroid $A\to M$, then all the (classical)
  $A$-Pontryagin classes of $E$ and $F$ must coincide.

  This paper generalises as well Bott's vanishing theorem to the
  setting of Lie algebroid representations (up to homotopy) on
  arbitrary vector bundles.  As an application, the main theorems
  induce new obstructions to the existence of infinitesimal ideal
  systems in a given Lie algebroid.

  \medskip
  
    \medskip

 
\end{abstract}

\maketitle\date{\today}

\tableofcontents

\section{Introduction}
Representations up to homotopy of Lie algebroids were found by Arias
Abad and Crainic \cite{ArCr12} to be a convenient geometric setting
for defining the adjoint representation of a Lie algebroid.  They
showed in \cite{ArCr11} that the \emph{adjoint representation up to
  homotopy} is the right notion of adjoint of a Lie algebroid since it
can be used to define its Weil algebra. The precursor notion of
\emph{strong homotopy representation} could be found already much
earlier in \cite{Stasheff88}, in the context of constrained Poisson
algebras -- incidentally, in the study of \emph{ideals} in constrained
Poisson algebras.  Further, $2$-term representations up to homotopy
are \emph{super-representations} in the sense of Quillen
\cite{Quillen85}.

Gracia-Saz and Mehta found in \cite{GrMe10a} that these
$2$-representations are equivalent to splittings of
VB-algebroids. This latter insight in particular led in the last ten
years to advances in the study of VB-algebroids with an additional
geometric structure -- see \cite{CaBrOr18}, \cite{Jotz18a,Jotz19a},
\cite{Jotz18b,Jotz18d}, \cite{GrJoMaMe18}, \cite{JoOr14},
\cite{DrJoOr15}, \cite{ShZh11} among others.  Representations up to
homotopy, in particular of $2$-representations, were further richly
studied in e.g.~\cite{ArSc13}, \cite{BrOr19}, \cite{Mehta15},
\cite{TrZh16}, \cite{ArCrDh11}, \cite{Jotz19b}, \cite{ArSc11}.

\medskip

\subsection*{Obstructions to the existence of $n$-representations}
Let us recall the definition of an \emph{$n$-term representation up to
  homotopy} \cite{ArCr12}, also called \emph{flat superconnection} in
\cite{GrMe10a}, but named here \emph{$n$-representation} for short.
\begin{mydef}[\cite{ArCr12,GrMe10a}]
  Let $A\to M$ be a Lie algebroid. Then an \textbf{$n$-representation}
  of $A$ is a graded vector bundle
  $\E=E_0[0]\oplus\ldots\oplus E_{n-1}[n-1]\to M$ with an operator
  \[\mathcal D\colon \Omega(A,\E)_\bullet\to\Omega(A,\E)_{\bullet+1}
  \]
  that increases the total degree by $1$ and satisfies $\mathcal D^2=0$ as well as 
  \begin{equation}\label{leibniz_n_conn}
    \mathcal D(\omega\wedge \eta)=\dr_A\omega\wedge \eta+ (-1)^{l}\omega\wedge\mathcal D\eta
    \end{equation}
  for $\omega\in\Omega^l(A)$ and
  $\eta\in\Omega(A,\E)_{\bullet}$.
\end{mydef}

An \textbf{$n$-connection} (or \emph{$n$-term connection up to
  homotopy}) \textbf{of a Lie algebroid $A$ on a graded vector bundle
  $\E=E_0[0]\oplus\ldots\oplus E_{n-1}[n-1]\to M$} is defined to be an
operator $\mathcal D$ as in the definition above, but without the
condition $\mathcal D^2=0$, see e.g.~\cite{Quillen85,GrMe10a,Mehta14}.

\medskip

The $A$-Pontryagin classes of a vector bundle $E$ measure ``the
failure of $E$ to have a flat $A$-connection'' -- or in other words to
carry a representation of $A$. Therefore, it is natural to ask if
there are \emph{characteristic classes of a graded vector bundle}
$\E=E_0[0]\oplus\ldots\oplus E_{n-1}[n-1]\to M$ that measure its
failure to carry an $n$-representation of a Lie algebroid $A$.

This paper explores the fact that the Chern-Weil construction of
Pontryagin characters carries over almost word by word to the setting
of $n$-connections, if the \emph{graded trace} on
$\underline{\End}(\E)$ replaces the trace on endomorphisms of an
ordinary vector bundle \cite{Quillen85,Mehta14}.  In short, given an
$n$-connection, its curvature $\mathcal D^2$ is (graded)
$\Omega^\bullet(A)$-linear and ``equals'' a form
$R_{\mathcal D}\in \Omega(A,\underline{\End}(\E))_\bullet$ of total
degree $2$. The graded trace $\widehat{\gtr}(R_{\mathcal D}^l)$ of the
$l$-th power of this form is just an element of $\Omega^{2l}(A)$, with
$\dr_A\left(\widehat{\gtr}(R_{\mathcal D}^l)\right)=0$, hence defining
a cohomology class
\[\left[\widehat{\gtr}(R_{\mathcal D}^l)\right]\in H^{2l}(A),
\]
called here the \textbf{$l$-th Pontryagin character of the graded
  vector bundle}.  These classes, for $l\geq 1$, do not depend on the
choice of the $n$-connection on $\E$, and they generate together the
\textbf{$A$-Pontryagin algebra of the graded vector bundle $\E$}, as
an $\R$-subalgebra of $H^\bullet(A)$.  Obviously the $A$-Pontryagin
algebra of $\E$ vanishes if $\E$ carries an $n$-representation of $A$.

A connection $\nabla\colon\Gamma(A)\times\Gamma(\E)\to\Gamma(\E)$ that
preserves the grading is an example of an $n$-connection of $A$ on
$\E$. Therefore the generators above of $\Pont_A^\bullet(\E)$ are
alternating sums of the classical Pontryagin
characters of the terms $E_i$ of $E$, $i=0, \ldots, n-1$. This
immediately yields the following theorem, which seems to have been
overlooked so far in the literature.

\begin{mythm}
  Let $\E=E_0[0]\oplus\ldots\oplus E_{n-1}[n-1]$ be a graded vector
  bundle over a smooth manifold $M$, and let $A\to M$ be a Lie
  algebroid.  If there exists an $n$-representation $\mathcal D$ of
  $A$ on $\E$, then the Pontryagin characters $\sigma_A^l(E_i)$,
  $l>1$, of the vector bundles $E_i$, $i=0,\ldots,n-1$, satisfy the
  equations
        \begin{equation}\label{gen_n_rep_case}
          \sum_{i=0}^{n-1}(-1)^i\sigma^l_A(E_i)=0\in H^{2l}(A)
        \end{equation}
        for all $l>1$.
      \end{mythm}
      In particular, for a graded vector bundle with grading
      concentrated in degrees $0$ and $1$, this theorem gives a simple
      obstruction to the existence of a $2$-representation (see
      Theorem \ref{main_app} below).
      \begin{mythm}\label{main_theorem_intro}
        Let $E$ and $F$ be smooth vector bundles over $M$, and let
        $A\to M$ be a Lie algebroid. If there is a $2$-representation
        of $A$ on $E[0]\oplus F[1]$, then the $A$-Pontryagin classes  of $E$ and $F$ are equal:
        \[ p^l_A(E)=p_A^l(F) \in H^l(A)
        \]
        for all $l\geq 1$.
      \end{mythm}
      Using the adjoint representation up to homotopy of a Lie
      algebroid $A\to M$, which is a $2$-representation of $A$ on
      $A[0]\oplus TM[1]$, this yields the following result.
\begin{mythm}
  Let $A$ be a vector bundle over a smooth manifold $M$, and let
  $\rho\colon A\to TM$ be a vector bundle morphism over the identity.
  If $A\to M$ carries a Lie algebroid structure with anchor $\rho$,
  then the Pontryagin classes of $A$ and $TM$ satisfy
        \[ \rho^\star \left(p^l(A)\right)=\rho^\star \left(p^l(TM) \right)\in H^l(A)
        \]
        for all $l\geq 1$.
      \end{mythm}
      This is in fact a special case of the following theorem, which is proved using a similar method.
       \begin{mythm}
        Let $A$ and $B$ be Lie algebroids over $M$. If there is a Lie
        algebroid morphism $\partial\colon B\to A$ over the identity
        on $M$, then
        \[\partial^*p_A^l(A)=p_B^l(A)=p_B^l(B)=\partial^*p_A^l(B)\in H^l(B)
        \]
        for all $l\geq 1$.
      \end{mythm}

\subsection*{Bott's vanishing theorems and obstructions to the
  existence of ideals in Lie algebroids}
The starting point of this paper is actually \emph{Bott's vanishing
theorem} \cite{Bott72} on Pontryagin classes and foliations:
\begin{thm}[\cite{Bott72}]
  Let $M$ be a smooth manifold and let $F_M$ be a subbundle of
  codimension $q$ of $TM$.  If $F_M$ is involutive, then the
  Pontryagin spaces
\[\Pont^l\left(TM/F_M\right)\subseteq  H^l(M)
  \]
  of $TM/F_M$ are all trivial for $l>2q$.
\end{thm}

Since an involutive subbundle $F_M\subseteq TM$ is always represented
on the normal bundle $TM/F_M$ via the Bott connection \cite{Bott72},
this theorem is a special case of the following result (see Theorem
\ref{main}).
\begin{mythm}\label{main_intro}
  Let $E$ be a smooth vector bundle over a smooth manifold $M$ and let
  $A$ be a Lie algebroid over $M$.  If there exists a Lie subalgebroid
  $B$ of $A$ of codimension $q$ with a linear representation
  $\nabla\colon \Gamma(B)\times\Gamma(E)\to\Gamma(E)$, then the
  $A$-Pontryagin spaces 
  \[\Pont^l_A(E)\subseteq H^l(A)
  \]
  are all trivial for $l>2q$.
\end{mythm}

The generalisation of this theorem to the setting of
Pontryagin algebras of a graded vector bundle is given by Theorem
\ref{Bott_gen_graded}. Although it does not yet lead to additional
obstruction results for particular examples, its proof is given in
detail for completeness and future applications.

\medskip

The author's original motivation for proving Theorem \ref{main_intro}
is her search for topological obstructions to the existence of ideals
in Lie algebroids. Jointly with Ortiz, the author identified in
\cite{JoOr14} what they consider the ``right notion'' of ideals in Lie
algebroids. These objects are called \emph{infinitesimal ideal
  systems} and defined as follows.

\begin{mydef}[\cite{JoOr14},\cite{Hawkins08}]\label{iis_def}
  Let $(q\colon A\to M, \rho,[\cdot\,,\cdot])$ be a Lie algebroid,
  $F_M\subseteq TM$ an involutive subbundle, $J\subseteq A$ a
  subbundle over $M$ such that $\rho(J)\subseteq F_M$, and $\nabla$ a
  flat partial $F_M$-connection on $A/J$ with the following
  properties:
\begin{enumerate}
\item If $a\in\Gamma(A)$ is $\nabla$-parallel\footnote{A section
    $a\in\Gamma(A)$ is said to be \textbf{$\nabla$-parallel} if
    $\nabla_X\bar a=0$ for all $X\in\Gamma(F_M)$. Here, $\bar a$ is
    the class of $a$ in $\Gamma(A/J)\simeq \Gamma(A)/\Gamma(J)$.},
  then $[a,j]\in\Gamma(J)$ for all $j\in\Gamma(J)$.
\item If $a,b\in\Gamma(A)$ are $\nabla$-parallel, then $[a,b]$ is also
  $\nabla$-parallel.\label{iis_def}
\item If $a\in\Gamma(A)$ is  $\nabla$-parallel, then $\rho(a)$ is
  $\nabla^{F_M}$-parallel, where
  \[\nabla^{F_M}\colon\Gamma(F_M)\times\Gamma(TM/F_M)\to\Gamma(TM/F_M)\]
  is the Bott connection associated to $F_M$.
\end{enumerate}
Then  the triple
$(F_M,J,\nabla)$ is an \textbf{infinitesimal ideal system in $A$.}
\end{mydef}
The first axiom implies immediately that $J\subseteq A$ is a
subalgebroid of $A$.  Infinitesimal ideal systems are an infinitesimal
version of the \emph{ideal systems} in \cite{HiMa90b,Mackenzie05} -- the
latter are exactly the kernels of fibrations of Lie algebroids.
Infinitesimal ideal systems already appear in \cite{Hawkins08} (not
under this name) in the context of geometric quantization as the
infinitesimal version of polarizations on groupoids.  Moreover, the
special case where $F_M=TM$ has been studied independently in
\cite{CrSaSt12} in relation with a modern approach to Cartan's work on
pseudogroups.

\medskip
 Consider an involutive subbundle $F_M\subseteq TM$ and the Bott
  connection\linebreak
$\nabla^{F_M}\colon\Gamma(F_M)\times\Gamma(TM/F_M)\to\Gamma(TM/F_M)$
associated to it. Then the triple $( F_M, F_M, \nabla^{F_M})$ is an
infinitesimal ideal system in the Lie algebroid $TM$. Therefore,
Bott's vanishing theorem provides an obstruction result for this
particular class of infinitesimal ideal systems. The general goal of
this paper is to find adequate generalisations of Bott's vanishing
theorem, yielding obstructions to the existence of infinitesimal ideal
systems in a given Lie algebroid $A\to M$ -- in terms of the
Pontryagin classes of $A$ and $TM$.

The following result (see Propositions \ref{prop_easy_ob1} and \ref{prop_easy_ob2})
gives the first set of information that can be extracted from
Theorem \ref{main_intro} and the definition of an infinitesimal ideal
system. 
\begin{proposition}
  Let $(F_M,J,\nabla)$ be an infinitesimal ideal system in a Lie
  algebroid $A\to M$. Let $s$ be the codimension of $J$ in $A$ and let
  $q$ be the codimension of $F_M$ in $TM$.  Then
  \begin{enumerate}
  \item the Pontryagin spaces $\Pont^l(A/J)$ and $\Pont^l(TM/F_M)$ in
    $H^\bullet(M)$ all vanish for $l>2q$, and
      \item the Pontryagin spaces
        $\Pont^l_A(A/J)$ and $\Pont^l_A(TM/F_M)$ in $H^\bullet(A)$ all vanish for $l>2\min\{s,q\}$.
        \end{enumerate}
\end{proposition}

However, this result turns out to be rather unsatisfactory on its own
because it uses only very few of the axioms of an infinitesimal ideal
system: (1), (2) and (3) in the definition 
are not used in the proof of this proposition. These three axioms
ensure \cite{DrJoOr15} that an infinitesimal ideal system in a Lie
algebroid $A\to M$ defines a subrepresentation $J[0]\oplus F_M[1]$ of
the adjoint representation up to homotopy of $A$ on
$A[0]\oplus TM[1]$, after the choice of a suitable connection.
Theorem \ref{main_theorem_intro} hence translates this fact in the
context of $A$-Pontryagin classes of $F_M$ and $J$. More precisely,
the results in \cite{DrJoOr15} and Theorem \ref{main_theorem_intro}
lead to further obstructions to the existence of an infinitesimal
ideal system in a Lie algebroid $A$ (see Theorem \ref{adv_ob_1}):
      \begin{mythm}
        Let $(A\to M, \rho, [\cdot\,,\cdot])$ be a Lie algebroid.
If $(F_M,J,\nabla)$ is an infinitesimal ideal system in $A$, then 
  \[ p^l_A(J)=p^l_A(F_M)
  \]
   for all $l\geq 1$.
\end{mythm}

\subsection*{Outline of the paper}
Section \ref{background} recalls in detail the Chern-Weil construction
of the Pontryagin classes of a vector bundle, using the powerful
modern language exposed in \cite{CrvdBa09}. The author recommends here
as well the reference \cite{Tu17}, which summarises in a beautiful
manner the construction of characteristic classes associated to vector
bundles and principal bundles, as well as some of their applications
in geometry and topology.

Section \ref{Bott_sec} proves the first generalisation of Bott's
vanishing theorem in \cite{Bott72}, and proves a refinement of it in
the case where an appropriate Atiyah class vanishes.

Section \ref{conn_hom_sec} studies connections up to homotopy and the
Pontryagin algebras of graded vector bundles of finite rank. The
obstruction to the existence of representations up to homotopy is also
proved there, as well as Bott's vanishing theorem for graded vector
bundles.

Section \ref{iis_sec} finally applies the prior results to the study
of characteristic classes defined by infinitesimal ideal systems in
Lie algebroids.

\subsection*{Outlook}
The construction of the $A$-Pontryagin algebra of a graded vector
bundle presented here can be extended to a construction of the
$(\mathcal M, \mathcal Q)$-Pontryagin algebra of a graded vector
bundle, for a Lie $n$-algebroid $(\mathcal M, \mathcal Q)$. This is
the subject of a project in progress that is joint with Papantonis.

\subsection*{Acknowledgements}
The author warmly thanks her student Jannick R\"onsch for writing his
Bachelor thesis about Section \ref{background_forms} and Theorem
\ref{main_intro}, which could already be found in an earlier draft of
this article. The review of his thesis, as well as the supervision
meetings, were a source of motivation for the author to continue her
preliminary work on this project.


Finally, the author thanks Theocharis Papantonis for his careful
reading, Thomas Schick for inspiring discussions, and Jim Stasheff for
useful comments.

\section{Preliminaries}\label{background}
This section recalls the modern definition of Pontryagin classes of a
vector bundle. It begins with some background on linear connections on
vector bundles and the associated calculus on differential forms. The
second subsection recalls the definition of the Pontryagin classes of
a vector bundle.  
In this section, the main reference
is \cite{CrvdBa09}, but $A$-Pontryagin classes of a vector bundle were
defined in \cite{Fernandes02}.

\subsection{Notation and vector-valued forms}

Given a Lie algebroid $A\to M$, we denote by $H^\bullet(A)$ the Lie
algebroid cohomology defined by the complex
$(\Omega^\bullet(A),\dr_A)$. As above, for simplicity, we write
$H^\bullet(M)$ for the (de Rham) cohomology of the standard Lie
algebroid $TM\to M$.

\medskip

Let $A$ be a Lie algebroid over a smooth manifold $M$ and let $E\to M$
be a smooth vector bundle. Then
$\Omega^\bullet(A,E):=\Gamma(\wedge^\bullet A^*\otimes E)$.  If $A=TM$
is the standard tangent Lie algebroid, then $\Omega^\bullet(TM,E)$ is
written $\Omega^\bullet(M,E)$ for simplicity. If $\underline{E}$ is a
graded vector bundle, then $\Omega(A,\E)_\bullet$ denotes
$\Gamma(\wedge^\bullet A^*\otimes \E)$ but with the total grading
defined as the sum of the form degree with the degree of the image in
$\E$.  The degree of a (degree-homogeneous) element $K$ of
$\Omega(A,\E)_\bullet$ is written $|K|\in\mathbb Z$.

For $K\in \Omega^l(A,\Hom(E,E'))$, the graded
$\Omega^\bullet(A)$-linear operator
$\widehat{K}\colon \Omega^\bullet(A,E)\to \Omega^{\bullet+l}(A,E')$ is
defined by $\widehat{K}(\omega)=K\wedge \omega$, i.e.
\[ \widehat{K}(\omega)(a_1,\ldots,a_{s+l})=\sum_{\sigma\in
    \mathfrak{S}_{(l,s)}}(-1)^{\sigma}K(a_{\sigma(1)},\ldots,a_{\sigma(l)})(\omega(a_{\sigma(l+1)},\ldots,
  a_{\sigma(l+s)}))
\]
for $\omega\in\Omega^s(A,E)$ and $a_1,\ldots,a_{s+l}\in\Gamma(A)$.
Here, $ \mathfrak{S}_{(l,s)}$ is the set of $(l,s)$-shuffles, i.e.~the permutations $\sigma\in S_{l+s}$ such that
$\sigma(1)<\ldots<\sigma(l)$ and $\sigma(l+1)<\ldots<\sigma(l+s)$.

The space of graded-$\Omega^\bullet(A)$-linear operators
$\Omega(A,E)\to \Omega(A,E')$ is denoted by
$\Hom^\bullet_{\Omega(A)}(\Omega(A,E),\Omega(A,E'))$. That is, an
element $\mathcal K$ of
$\Hom^s_{\Omega(A)}(\Omega(A,E),\Omega(A,E'))$, for $s\geq 0$, is a
map $\mathcal K\colon\Omega^\bullet(A,E)\to \Omega^{\bullet+s}(A,E')$
satisfying
$\mathcal
K(\omega\wedge\eta)=(-1)^{s\cdot|\omega|}\omega\wedge\mathcal K(\eta)$
for all $\omega\in\Omega^\bullet(A)$ and $\eta\in\Omega^\bullet(A,E)$.

The map
$\Omega^\bullet(A,\Hom(E,E'))\to
\Hom^\bullet_{\Omega(A)}(\Omega(A,E),\Omega(A,E'))$ given by
$K\mapsto \widehat{K}$ is a bijection \cite{ArCr12}, with inverse sending
$\mathcal K\colon \Omega^\bullet(A,E)\to\Omega^{\bullet+s}(A,E')$ to
$\mathcal K_0\in\Omega^s(A,\Hom(E,E'))$ defined
by
\[\mathcal K_0(a_1,\ldots,a_s)(e)=\mathcal
  K(e)(a_1,\ldots,a_s)\] for $a_1,\ldots,a_s\in\Gamma(A)$ and
$e\in\Gamma(E)$.

Let now $\E=\oplus_{z\in\mathbb Z}E_z[z]$ be a graded vector bundle
over $M$.  As always, the $\Omega^\bullet(A)$-module of $\E$-valued
forms $\Omega^\bullet(A,\E)$ has a total grading given by
$\deg \eta=s+l$ for $\eta\in\Omega^s(A,E_l)$.
Here also, there is a bijection between elements
$K\in \Omega(A,\underline{\Hom}(\E,\underline{F}))_s$ and
graded-$\Omega^\bullet(A)$-linear operators
$\mathcal K\colon
\Omega(A,\E)_\bullet\to\Omega(A,\underline{F})_{\bullet+s}
$ that increase the total degree by $s$.  An element 
$K\in \Omega(A,\underline{\Hom}(\E,\underline{F}))_s$
can be written
\[K=\sum_{i=0}^s\sum_{j-l=s-i}K^{i,l,j}\in\bigoplus_{i=0}^s\bigoplus_{j-l=s-i}\Omega^i(A,\Hom(E_l,F_j)).
\]
The
corresponding
$\widehat{K}\in\End_{\Omega^\bullet(A)}(\Omega(A,\E),\Omega(A,\E))_s$
is given by
\[ \widehat{K}=\sum_{i=0}^s\sum_{j-l=s-i}\widehat{K^{i,l,j}},
\]
with
$\widehat{K^{i,l,j}}\colon \Omega^\bullet(A,E_l)\to
\Omega^{\bullet+i}(A,F_j)$ defined as before. The inverse to the map
$\widehat{\cdot}$ is easily defined as above.

Finally, the graded commutator of degree-homogeneous elements\linebreak
$K_1,K_2\in\Omega(A, \underline{\End}(\E))_\bullet $
can now be defined by
\[
  \widehat{[K_1,K_2]}=\left[\widehat{K_1},\widehat{K_2}\right]=\widehat{K_1}\circ\widehat{K_2}-(-1)^{|K_1|\cdot|K_2|}\widehat{K_2}\circ\widehat{K_1}.
\]

\subsection{Linear connections on vector bundles, and vector valued forms}\label{background_forms}
Let $E\to M$ be a vector bundle, and let
$(A\to M, \rho, [\cdot\,,\cdot])$ be a Lie algebroid over the same
base.  Then a linear $A$-connection
$\nabla\colon \Gamma(A)\times\Gamma(E)\to\Gamma(E)$ is equivalent to
an operator
$\dr_\nabla\colon \Omega^\bullet(A,E)\to \Omega^{\bullet+1}(A,E)$
satisfying
\begin{equation}\label{connection}
  \dr_\nabla(\omega\wedge\eta)=(\dr_A\omega)\wedge \eta+(-1)^l\omega\wedge\dr_\nabla\eta
\end{equation}
for $\omega\in \Omega^l(A)$ and $\eta\in\Omega^\bullet(A,E)$.
Given $\nabla$, the operator $\dr_\nabla$ is defined by \eqref{connection} and by
 \[\dr_\nabla e=\nabla_\cdot e\in\Omega^1(A,E)\]
 for $e\in\Gamma(E)=\Omega^0(A,E)$.  For instance, if $E=\R\times M$
 with the canonical flat $A$-connection $\nabla_af=\ldr{\rho(a)}(f)$,
 then $\Omega^\bullet(A,E)\simeq \Omega^\bullet(A)$ and
 $\dr_\nabla=:\dr_A$, which satisfies in addition $\dr_A^2=0$ and
 defines the Lie algebroid cohomology $H^\bullet(A)$.
In general,
\[ \dr_\nabla^2=\widehat{R_\nabla}\colon \Omega^\bullet(A,E)\to \Omega^{\bullet+2}(A,E).
\]

Let $E$ and $E'$ be vector bundles over $M$, and let $\nabla$ and
$\nabla'$ be linear $A$-connections on $E$ and $E'$, respectively. The
reader is invited to check (see also \cite{CrvdBa09}) that for
$K\in\Omega^s(A,\Hom(E,E'))$,
\begin{equation}\label{d_nabla_K_formula}
  \dr_{\nabla'} \circ\widehat{K}-(-1)^s\widehat{K}\circ\dr_\nabla=\widehat{\dr_{\nabla^{\Hom}}K},
  \end{equation}
  where
  $\nabla^{\Hom}\colon
  \Gamma(A)\times\Gamma(\Hom(E,E'))\to\Gamma(\Hom(E,E'))$ is defined
  by
  $\left(\nabla^{\Hom}_a\phi\right)(e)=\nabla'_a(\phi(e))-\phi(\nabla_ae)$
  for $a\in\Gamma(A)$ and $e\in\Gamma(E)$.  If $E=E'$ and
  $\nabla=\nabla'$, set

  \begin{equation}\label{d_nabla_K_end}
    \left[\dr_\nabla, \widehat{K}\right]:=\dr_{\nabla}
    \circ\widehat{K}-(-1)^k\widehat{K}\circ\dr_\nabla=\widehat{\dr_{\nabla^{\End}}K}.
    \end{equation}

    \medskip The trace operator
    $\tr\colon\Gamma(\End(E))\to C^\infty(M)$ can be understood as an
    element of $\Omega^0(A,\Hom(\End(E),\mathbb R))$, and so defines
    as above an $\Omega^\bullet(A)$-linear map
    $\widehat{\tr}\colon
    \Omega^\bullet(A,\End(E))\to\Omega^\bullet(A)$ that preserves the
    degree.

    Equip $\mathbb R\times M$ as above with the flat $A$-connection
    $\ldr{}\colon \Gamma(A)\times C^\infty(M)\to C^\infty(M)$, and the
    vector bundle $\End(E)$ with the connection induced by
    $\nabla\colon \Gamma(A)\times\Gamma(E)\to\Gamma(E)$. Then the
    induced connection
\[\nabla^{\Hom}\colon \Gamma(A)\times \Gamma(\Hom(\End(E),\mathbb R))\to \Gamma(\Hom(\End(E),\mathbb R))
\]
applied to the trace operator reads
\[ (\nabla^{\Hom}_a\tr)(\phi)=\ldr{\rho(a)}(\tr(\phi))-\tr\left(\nabla_a^{\End}\phi\right)
\]
for $a\in\Gamma(A)$ and $\phi\in\Gamma(\End(E))$.

\begin{lemma}\label{lemma_easy_trace_flat}
  With the choices of connections above,  $\nabla^{\Hom}_a\tr=0$ for all $a\in\Gamma(A)$.
  \end{lemma}
 
  \begin{proof}
    Take a local frame $(e_1,\ldots,e_k)$ of
$E$ over an open set $U\subseteq M$ and consider the dual local frame
$(\epsilon_1,\ldots,\epsilon_k)$ of $E^*$. It is easy to see that for
each $i,j=1,\ldots,k$ and each $a\in\Gamma_U(A)$:
\begin{equation*}
  \begin{split}
    (\nabla^{\Hom}_a\tr)(e_i\otimes \epsilon_j)&=\ldr{\rho(a)}(\delta_{ij})-\tr(\nabla_a^{\End}(e_i\otimes \epsilon_j))=-\sum_{s=1}^k\langle \epsilon_s,\nabla_a^{\End} (e_i\otimes \epsilon_j)(e_s)\rangle\\
    &=-\sum_{s=1}^k\bigl\langle \epsilon_s, \nabla_a(\delta_{js}e_i)-e_i\langle\epsilon_j, \nabla_ae_s\rangle\bigr\rangle\\
    &=-\langle \epsilon_j, \nabla_ae_i\rangle +\langle\epsilon_j,
    \nabla_ae_i\rangle=0.
\end{split}
\end{equation*}
\end{proof}

Lemma \ref{lemma_easy_trace_flat} and \eqref{d_nabla_K_formula} yield
the equality
\begin{equation}\label{tr_dA_formula}
  \dr_A\circ\widehat{\tr}=\widehat{\tr}\circ \dr_{\nabla^{\End}} \colon \Omega^\bullet(A, \End(E))\to \Omega^{\bullet+1}(A).
\end{equation}

  \subsection{$A$-Pontryagin characters of a vector bundle}\label{classical_pont}
  As before, consider a Lie algebroid $A\to M$, and a vector bundle
  $E\to M$ of rank $k$, with a linear $A$-connection
  $\nabla\colon\Gamma(A)\times\Gamma(E)\to\Gamma(E)$.  Let
  $R_\nabla\in\Omega^2(A,\Hom(TM,A))$ be the curvature of $\nabla$.
  
  Define for $i\geq 1$ the form $R_\nabla^i\in\Omega^{2i}(A,\End(E))$
  by
  \[
    \widehat{R_\nabla^i}=\widehat{R_\nabla}^i=\dr_\nabla^{2i}\in\End_{\Omega^\bullet(A)}(\Omega^\bullet(A,E)). \]
  Then \eqref{d_nabla_K_end} shows
  $\widehat{\dr_{\nabla^{\End}} R_\nabla^i}=\left[\dr_\nabla,
    \widehat{R_\nabla^i}\right]=\left[\dr_\nabla,
    \widehat{R_\nabla}^i\right]=\left[\dr_\nabla,
    \dr_\nabla^{2i}\right]=0$, and so with \eqref{tr_dA_formula}:
  \begin{equation}\label{closed_easy}
    \dr_A(\widehat{\rm tr}(R_\nabla^i))=\widehat{\rm tr}(\dr_{\nabla^{\End}}R_\nabla^i)=0.
    \end{equation}
    Therefore, $\widehat{\rm tr}(R_\nabla^i)$ defines a cohomology class in $H^{2i}(A)$.

    \begin{lemma}
      Let $E\to M$ be a vector bundle and let $A\to M$ be a Lie
      algebroid. Then the chomology class
      $\left[\widehat{\rm tr}(R_\nabla^i)\right]\in H^{2i}(A)$ does
      not depend on the choice of $A$-connection $\nabla$ on $E$, for
      $i\geq 1$.
    \end{lemma}

    This proof is standard; in the context of Lie algebroid Pontryagin
    classes, it is due to \cite{Fernandes02} following a classical
    method.  The proof is omitted here, but done later in the more
    general setting of Pontryagin algebras defined by connections up
    to homotopy (see Proposition \ref{inv_conn}, and Appendix
    \ref{proof_inv_conn}); in the same manner as in \cite{Quillen85}
    for superconnections.

  \begin{definition}
    Let $E$ be a vector bundle over $M$ and let $A\to M$ be a Lie
    algebroid.
    \begin{enumerate}
      \item Choose any linear $A$-connection $\nabla$ on $E$. The
    cohomology classes
    $\sigma_A^i(E):=\left[\widehat{\rm tr}(R_\nabla^i)\right]\in H^{2i}(A)$, for
    $i\geq 1$, are called the \textbf{$A$-Pontryagin characters} of $E$.
  \item The \textbf{$A$-Pontryagin algebra} of $E$ is the
    $\R$-subalgebra $\Pont_A^\bullet(E)\subseteq H^\bullet(A)$
    generated by the $A$-Pontryagin characters.
    \end{enumerate}
\end{definition}

The Pontryagin algebra is also called the \emph{characteristic
  algebra} in \cite{Tu17}.  It is easy to see that $\Pont_A^l(E)=0$
for $l$ an odd number.  It is a standard fact that even $\Pont_A^l=0$
for $l$ not divisible by $4$. For completeness, Bott's proof of this
fact \cite{Bott72} is quickly recalled here.  Equip the vector bundle
$E$ with a smooth metric (i.e.~a positive definite fibrewise pairing),
and take the $A$-connection
$\nabla\colon \Gamma(A)\times\Gamma(E)\to\Gamma(E)$ to be
\emph{metric}:
$\langle\nabla_ae, e'\rangle+\langle e, \nabla_ae'\rangle=\ldr{\rho(a)}\langle e, e'\rangle$
for $a\in\Gamma(A)$, $e,e'\in\Gamma(E)$. Then it is easy to check that
$\langle R_\nabla(a,b)e, e'\rangle=-\langle e, R_\nabla(a,b)e'\rangle$
for all $a,b\in\Gamma(A)$, $e,e'\in\Gamma(E)$,
and inductively
\[ \langle R_\nabla^i(a_1,b_1, a_2,b_2, \ldots, a_i,b_i)e, e'\rangle=(-1)^{i}\langle e, R_\nabla^i(a_1,b_1, a_2,b_2, \ldots, a_i,b_i)e'\rangle
\]
for $i\geq 1$.  Then immediately $\widehat{\tr}(R_\nabla^i)=0$ for $i$
odd, and so $\Pont^{2i}_A(E)=0$ for $i$ odd.

\bigskip

Finally, the \textbf{$A$-Pontryagin classes of the vector bundle $E$}
can be defined; see e.g.~\cite{Tu17} for detailed explanations.  Consider
$\operatorname{Gl}(k,\R)$-invariant polynomial functions
$p\colon \lie{gl}(k,\R)\to\R$, i.e.~such that for all
$g\in \operatorname{Gl}(k,\R)$ and $X\in \lie{gl}(k,\R)$
\[p(gXg\inv)=p(X).
\]
The $\operatorname{Gl}(k,\R)$-invariant polynomials on
$\lie{gl}(k,\R)$ form an $\R$-algebra, which is generated as an
$\R$-algebra by the polynomials $\Sigma_0,\Sigma_1,\ldots$ defined by
\[ \Sigma_i(X)=\operatorname{trace}(X^i)
\]
for all $X\in\lie{gl}(k,\R)$ (see for instance \cite{Bott72}).
Each of these polynomials defines the
cohomology class 
\begin{equation}\label{generators}
  \sigma^i_A(E):=\left[\Sigma_i(R_\nabla)\right]:=\left[\widehat{\tr}(R_\nabla^i)\right] \in H^{2i}(A),
\end{equation}
called the \textbf{$i$-th Pontryagin character of $E$}.  As a
consequence, each $\operatorname{Gl}(k,\R)$-invariant polynomial $p$
on $\lie{gl}(k,\R)$ defines a closed form
$p(R_\nabla)\in \Omega^\bullet(A)$ and an element
$[p(R_\nabla)]\in H^\bullet(A)$. More precisely, if
$p=q(\Sigma_{i_1},\ldots,\Sigma_{i_l})\in
\R[\Sigma_1,\Sigma_2,\ldots]$, then
\[p(R_\nabla)=q(\Sigma_{i_1}(R_\nabla), \ldots,\Sigma_{i_l}(R_\nabla)).
\]
For instance, $p=\Sigma_2-(\Sigma_1)^2$ gives
$p(R_\nabla)=\widehat{\tr}(R_\nabla^2)-\widehat{\tr}(R_\nabla)\wedge
\widehat{\tr}(R_\nabla)$.  This defines
the Chern-Weil morphism of $\R$-algebras
\[ \rm{cw}_A(E)\colon
  \operatorname{Sym}^\bullet(\lie{gl}(k,\R))^{\operatorname{Gl}(k,\R)}\to
  H^{2\bullet}(A), \qquad p\mapsto \left[p(R_\nabla)\right].
  \]
  The $\R$-subalgebra $\Pont_A^\bullet(E)\subseteq H^\bullet(A)$ is
  the image of this morphism, i.e.~the subalgebra of all cohomology
  classes $[p(R_\nabla)]$ defined by
  $\operatorname{Gl}(k,\R)$-invariant polynomial $p$ on
  $\lie{gl}(k,\R)$.

For $i$ a positive integer, 
the characteristic polynomial
\begin{equation}\label{def_ps}
  \det\left(\lambda\cdot I_k  +X\right)=\sum_{i=0}^kf_{i}(X)\lambda^{k-i}
\end{equation} 
defines homogeneous polynomials $f_{i}$ of degree $i$ on
$\mathfrak{gl}(k,\R)$, for $k\geq i\geq 0$. These polynomials are
obviously $\operatorname{Gl}(k,\R)$-invariant, and so for each
$i\geq 1$, the \textbf{$i$-th $A$-Pontryagin class of $E$} can be defined as
\[ p^{i}_A(E):=\left[f_{2i}\left(\frac{i}{2\pi}R_\nabla\right)\right]\in H^{4i}(A),
\]
for any choice of connection
$\nabla\colon\Gamma(A)\times\Gamma(E)\to\Gamma(E)$. The $A$-Pontryagin
classes of $E$ generate together $\Pont_A^\bullet(E)$ (see for
instance \cite{Bott72}).  The \textbf{total $A$-Pontryagin class of
  $E$} is defined by
\[ p_A(E)=\left[
    \det\left(I_k+\frac{i}{2\pi}R_\nabla\right)\right]=1+p_A^1(E)+p_A^2(E)+\ldots
  +p^{\lfloor \frac{k}{2}\rfloor}\in \Pont^\bullet_A(E).
  \]

  \begin{remark}\label{imp_rem}
    Given an ordinary linear connection
    $\nabla\colon\mx(M)\times \Gamma(E)\to\Gamma(E)$ on a vector
    bundle $E$ of rank $k$, a Lie algebroid $A\to M$ defines
    a linear $A$-connection $\nabla^A\colon\Gamma(A)\times\Gamma(E)\to\Gamma(E)$ by $\nabla^A_ae=\nabla_{\rho(a)}e$.
    It is easy to see that
    \[ [p(R_{\nabla^A})]=\rho^\star[p(R_\nabla)]\in H^\bullet(A)
    \]
    for any $\operatorname{Gl}(k,\R)$-invariant polynomial $p$ on
    $\lie{gl}(k,\R)$. Here, $\rho^\star$ is the cochain map
    \[\rho^\star\colon (\Omega^\bullet(M),
    \dr)\to(\Omega^\bullet(A),\dr_A),\]
    $\rho^\star(\omega)(a_1,\ldots,
    a_s)=\omega(\rho(a_1),\ldots,\rho(a_s))$ for
    $\omega\in\Omega^s(M)$ and $a_1,\ldots,a_s\in\Gamma(A)$.

    As observed by Fernandes in \cite{Fernandes02}, this yields
    $\Pont^\bullet_A(E)=\rho^\star(\Pont^\bullet(E))$,
    or more precisely $\rm{cw}_A(E)=\rho^\star\circ \rm{cw}(E)$.
    \end{remark}

  \section{Bott's vanishing theorem in a more general setting}\label{Bott_sec}
  This section rephrases Bott's proof of the vanishing Pontryagin
  classes of the normal bundle to an involutive subbundle of the
  tangent \cite{Bott72}. Since the decisive object is a flat
  $F_M$-connection on a smooth vector bundle $TM/F_M$, that can be
  extended to a linear $TM$-connection in order to define Pontryagin
  characters or classes, one can easily prove a similar result for the existence of
  a flat partial connection on a smooth vector bundle. Further, the
  construction is adapted to the more general $A$-Pontryagin classes
  of a vector bundle $E$.

\subsection{Bott's vanishing theorem}
Let $A$ be a Lie algebroid over a smooth manifold $M$, and let $B$ be
a subalgebroid of $A$ over $M$.  Let $n$ be the rank of $A$, and $l$
be the rank of $B$.  Set $q:=n-l$, the rank of $A/B$, or of the
annihilator $B^\circ$ of $B$. Let $E$ be a smooth vector bundle over
$M$, with a \emph{flat $B$-connection $\nabla$}.  It is not difficult
to see that $\nabla$ can be \emph{extended} to an $A$-connection
$\tilde\nabla\colon\Gamma(A)\times\Gamma(E)\to\Gamma(E)$, satisfying
\begin{equation}\label{nabla_ext}
  \tilde\nabla_be=\nabla_be
\end{equation}
for all $b\in\Gamma(B)$ and $e\in\Gamma(E)$.

\medskip

Define the space $I^\bullet(B)\subseteq \Omega^\bullet(A)$ as the
ideal in $\Omega(A)$ generated by the $1$-forms vanishing on $B$. That
is, it is generated by the sections of the \emph{annihilator}
$B^\circ\subseteq A^*$ of $B$. It is explicitly given by
$I^0(B)=\{0\}\subseteq \Omega^0(A)=C^\infty(M)$ and 
\[  I^r(B)=\{\omega\in\Omega^r(A)\mid \omega(b_1,\ldots,b_r)=0 \text{ for all } b_1,\ldots,b_r\in\Gamma(B)
  \}
\]
for $r\geq 1$.

Choose an open set $U\subseteq M$ trivialising $A$ and $B$. That is,
there is a smooth frame $(a_1,\ldots,a_n)$ for $A$ over $U$ such that
$(a_{q+1},\ldots,a_n)$ is a smooth frame for $B$. Consider the dual
frame $(\alpha_1,\ldots,\alpha_n)$ of $A^*$ over $U$. By construction,
$(\alpha_{1},\ldots,\alpha_q)$ is a smooth frame for $B^\circ$ over
$U$.  Since $I^\bullet(B)$ is generated as an ideal by $\Gamma(B^\circ)$, for
$r\geq 1$, an element $\omega$ of $I^r_U(B)$ can be written as
\[\omega=\sum_{i=1}^q\omega_i\wedge\alpha_i 
\]
with $\omega_i\in\Omega^{r-1}_U(A)$. Therefore, since $B^\circ$ has
rank $q$, the wedge product
\[(I^\bullet(B))^{q+1}=\underset{ q+1 \text{
    times}}{\underbrace{I^\bullet(B)\wedge\ldots\wedge I^\bullet(B)}}\]
must necessarily vanish.

It is  now easy to see that \eqref{nabla_ext} implies
\[R_{\tilde\nabla}(b,b')e=R_{\nabla}(b,b')e= 0\] for
$b,b'\in\Gamma(B)$ and all $e\in\Gamma(E)$, and so
$R_{\tilde\nabla}\in I^2(B)\otimes_{C^\infty(M)}\Gamma(\End(E))$.
This implies $R_{\tilde\nabla}^i\in (I^2(B))^i\otimes_{C^\infty(M)}\Gamma(\End(E))$ and so
$\widehat{\tr}(R_{\tilde\nabla}^i)\in (I^2(B))^i$.
More generally, for $p$ a $\operatorname{Gl}(k,\R)$-invariant
polynomial of degree $d$ on $\lie{gl}(k,\R)$, the $2d$-form
$p(R_\nabla)\in\Omega^{2d}(A)$ is an element of $(I^2(B))^d$ and so
$p(R_\nabla)=0$ for $d>q$.

As a summary, this section has proved the following result.
\begin{theorem}\label{main}
  Let $E$ be a smooth vector bundle over a smooth manifold $M$ and let
  $A$ be a Lie algebroid over $M$.  If there exists a Lie subalgebroid
  $B$ of $A$ of codimension $q$ with a linear representation
  $\nabla\colon \Gamma(B)\times\Gamma(E)\to\Gamma(E)$, then the Pontryagin spaces
  \[\Pont^l_A(E)\subseteq H^l(A)
  \]
  are all trivial for $l>2q$.
\end{theorem}

Using Remark \ref{imp_rem}, this yields the following obstruction
result in terms of the classical Pontryagin spaces of $E$.
\begin{corollary}
  Let $E$ be a smooth vector bundle over a smooth manifold $M$ and let
  $A$ be a Lie algebroid over $M$.  If there exists a Lie subalgebroid
  $B$ of $A$ of codimension $q$ with a linear representation
  $\nabla\colon \Gamma(B)\times\Gamma(E)\to\Gamma(E)$, then the Pontryagin spaces
  \[\Pont^l(E)\subseteq H^l(M)
  \]
  all lie in the kernel of $\rho^\star\colon H^\bullet(M)\to H^\bullet(A)$ for $l>2q$.
\end{corollary}

If a Lie algebroid $A$ has a subalgebroid $B$ of
codimension $q$; then $B$ is
represented on $A/B$ via the flat Bott-connection
\[\nabla^B\colon \Gamma(B)\times\Gamma(A/B)\to\Gamma(A/B), \quad
  \nabla^B_b\bar a=\overline{[b,a]}.
\]
Hence $\Pont^l_A(A/B)\subseteq H^l(A)$ is trivial for $l>2q$. This
yields obstructions to a subalgebroid structure on $B\subseteq A$ of
codimension $q$.

However, in the case $A=TM$ and $B=F_M$, the algebroid $F_M$ is in
fact more than just a subalgebroid: it carries as well an infinitesimal ideal
system \cite{JoOr14}. The goal of this paper is the generalisation of Bott's
vanishing theorem \cite{Bott72} as a statement on ideals.

\subsubsection{Massey products}\label{massey1}
As already emphasised in \cite{Bott72}, Theorem \ref{main} shows more
than the vanishing of the Pontryagin classes $p_A^l(E)$ for $l>2q$. It
shows the vanishing of all $A$-characteristic classes of $E$ defined
by invariant polynomials of degree $d>q$.  In \cite{Bott72}, Bott
follows an idea of Shulman in order to refine his theorem and express
this fact. For completeness, this is quickly discussed here in the
more general setting of this paper.

     Let $A\to M$ be a Lie algebroid and
     $[\alpha], [\beta], [\gamma]\in H^\bullet(A)$ be classes such
     that
\[ [\alpha]\wedge[\beta]=0 \quad \text{ and } \quad [\beta]\wedge[\gamma]=0.
\]
Then $\alpha\wedge\beta=\dr_A\omega$ and
$\beta\wedge\gamma=(-1)^{|\alpha|}\dr_A\eta$ for some forms $\omega$
and $\eta\in\Omega^\bullet(A)$. As a consequence,
$\dr_A(\omega\wedge\gamma)=\alpha\wedge\beta\wedge\gamma=\dr_A(\alpha\wedge\eta)$,
which shows that the class
\[
  \langle [\alpha], [\beta], [\gamma]\rangle:=[\omega\wedge
  \gamma-\alpha\wedge\eta]\in H^\bullet(A)\] is defined. As mentioned
in \cite{Bott72}, this is called the \textbf{Massey triple product}
\cite{Massey58} of $[\alpha], [\beta], [\gamma]\in H^\bullet(A)$; it
is well-defined up to an element of
$I^\bullet([\alpha],[\gamma])\subseteq H^\bullet(A)$, the ideal
generated by $[\alpha]$ and $[\gamma]$.  

Consider the situation of Theorem \ref{main} and take any three
classes $[\alpha]$, $[\beta]$ and $[\gamma]$ in $\Pont^\bullet_A(E)$
such that $|\alpha|+|\beta|>2q $ and $|\beta|+|\gamma|>2q$.  Then
$\alpha, \beta,\gamma\in\Omega^\bullet(A)$ can be chosen
$\alpha=p_\alpha(R_\nabla )$, $\beta=p_\beta(R_\nabla )$ and
$\gamma=p_\gamma(R_\nabla )$ for $\nabla$ as in the proof of Theorem
\ref{main} and $p_\alpha,p_\beta, p_\gamma$
$\operatorname{Gl}(k,\R)$-invariant polynomials on $\lie{gl}(k,\R)$ of
degrees $|\alpha|/2$, $|\beta|/2$ and $|\gamma|/2$, respectively --
where $k$ is the rank of $E$. Then by definition,
$\alpha\wedge\beta=(p_\alpha\cdot p_\beta)(R_\nabla)$, which must
vanish by the proof of Theorem \ref{main} and $|\alpha|+|\beta|>2q $,
and in the same manner $\beta\wedge\gamma=0$.  Then by definition,
$\langle [\alpha], [\beta], [\gamma]\rangle=0$.  This proves the
following theorem, which is attributed to Shulman in \cite{Bott72}.

\begin{theorem}
  Let $E$ be a smooth vector bundle over a smooth manifold $M$ and let
  $A$ be a Lie algebroid over $M$.  If there exists a Lie subalgebroid
  $B$ of $A$ of codimension $q$ with a linear representation
  $\nabla\colon \Gamma(B)\times\Gamma(E)\to\Gamma(E)$, then for all
  $[\alpha]$, $[\beta]$ and $[\gamma]$ in $\Pont^\bullet_A(E)$ such
  that $|\alpha|+|\beta|>2q $ and $|\beta|+|\gamma|>2q$,
  \[ \langle [\alpha], [\beta], [\gamma]\rangle=0.
    \]
  \end{theorem}

\subsection{Reducible vector bundles -- a short discussion}\label{reductible_vb}

  Consider a fibration of vector bundles
  \begin{equation*}
\begin{xy}
  \xymatrix{
    E\ar[d]_{q_E}\ar[r]^{\phi}&E'\ar[d]^{q_{E'}}\\
    M\ar[r]_{f}&M' }
\end{xy}
\end{equation*}
i.e.~a fibrewise surjective vector bundle morphism $\phi$ over a smooth
surjective submersion $f$. Assume that $E$ and $E'$ have the same
rank, so that $\phi$ restricted to each fibre is a bijection.  If $f$
has connected fibres, then $M'$ can be identified with the leaf space
of the involutive subbundle $T^fM:=\ker(Tf)\subseteq TM$ and the
morphism $\phi$ defines a flat $T^fM$-connection \cite{JoOr14}
$\nabla\colon \Gamma(T^fM)\times\Gamma(E)\to \Gamma(E)$ by
\[
  \nabla_Xe=0 \,\text{ for all } \, X\in\Gamma(T^fM) \quad
  :\Leftrightarrow \quad \exists\, e'\in\Gamma(E'): \phi\circ e=e'\circ
  f.\] That is, the $\nabla$-flat sections of $E$ are the sections of
$E$ that are $\phi$-projectable to sections of $E'$.

Conversely, consider a smooth vector bundle $E\to M$, an involutive
subbundle $F_M\subseteq TM$ and a flat connection
$\nabla\colon\Gamma(F_M)\times\Gamma(E)\to\Gamma(E)$.  If $F_M$ is
simple and $\nabla$ has no holonomy, then they induce a fibration of
vector bundles \cite{JoOr14}
\begin{equation*}
\begin{xy}
  \xymatrix{
    E\ar[d]_{q_E}\ar[r]^{\pi}&E/\nabla\ar[d]^{[q_{E}]}\\
    M\ar[r]_{\pi_M}&M/F_M }
\end{xy}
\end{equation*}
where $E/\nabla$ is the quotient of $E$ by parallel transport.

\medskip We say that $E$ is $q$-reducible if there is a fibration of
vector bundles
  \begin{equation*}
  \begin{xy}
    \xymatrix{
      E\ar[d]_{q_E}\ar[r]^{\phi}&E'\ar[d]^{q_{E'}}\\
      M\ar[r]_{f}&M' }
\end{xy}
\end{equation*}
such that
$\dim M'=q$ and
$\rank E=\rank E'$. Then the Pontryagin classes of $E$ of degree
greater than $q$ must necessarily vanish. This is because the Gauss
map $g_E$ of $E$ then factors as $g_E=g_{E'}\circ f$, and so
$\Pont^\bullet(E)=f^*\Pont^\bullet(E')$.  Therefore, in that case,
Bott's vanishing theorem (Theorem \ref{main}) is satisfied even with
$q$ as lower bound.

Pontryagin classes are invariants of a vector bundle, that vanish if
it is trivializable.  In particular, the Pontryagin classes of $E$ of
rank $k$ all vanish if there is a smooth morphism of vector bundles
\begin{equation*}
\begin{xy}
  \xymatrix{
    E\ar[d]_{q_E}\ar[r]^{\pi}&\R^k\ar[d]\\
    M\ar[r]&\{{\rm pt}\} }
\end{xy}
\end{equation*}
that restricts to an isomorphism on each fibre.  The consideration
above shows that much finer geometrical information can be extracted
from Pontryagin classes, and that they could be seen as obstructions
to (constant rank) fibrations to low dimensional manifolds. For
instance, if $\Pont^l(E)\neq \{0\}$ for some $l\geq 4$, then the
vector bundle $E$ is not $1$-reducible, and if $\Pont^l(E)\neq \{0\}$
for some $l>4$, then the vector bundle $E$ is not $2$-reducible, etc.

\subsection{Bott's vanishing theorem and the Atiyah class}
If $E\to M$ has a flat $F_M$-connection, but $F_M$ is not simple or
the holonomy of $\nabla$ is not trivial, then the vector bundle $E$
still is ``infinitesimally symmetric along $F_M$'', but we can only
prove Bott's vanishing theorem with lower bound $2q$. However,
following ideas by Molino \cite{Molino71a} (see also \cite{KaTo75}),
Theorem \ref{main} holds with the lower bound $q$ instead of $2q$ if
the \emph{Atiyah class} of the connection vanishes. On the other hand,
the new, more general version of Bott's vanishing theorem in Theorem
\ref{main}, might be useful in the search for examples where $E$ has a
flat $F_M$-connection, with $F_M$ of codimension $q$, but  its
$k$-th Pontryagin class does not vanish for some $k>\frac{q}{2}$.

\medskip

Let $A$ be a Lie algebroid over a smooth manifold $M$, and let $B$ be
a subalgebroid of $A$ over $M$, of codimension $q$. Let $E$ be a
smooth vector bundle over $M$, as before with a flat $B$-connection
$\nabla$.  Take again an extension
$\tilde\nabla\colon\Gamma(A)\times\Gamma(E)\to\Gamma(E)$ of $\nabla$
as in \eqref{nabla_ext}. Then the form
$\omega_{\tilde\nabla}\in\Omega^1(B,\Hom(A/B,\End(E)))$ is defined by 
\[\omega_{\tilde\nabla}(b,\bar a)(e)=R_{\tilde\nabla}(b,a)e.
\]
The flat $B$-connection $\nabla$ on $E$ and the flat Bott-connection
$\nabla^{B}\colon\Gamma(B)\times\Gamma(A/B)\to\Gamma(A/B)$ combine to
a flat $B$-connection $\nabla^{\Hom}$ on $\Hom(A/B,\End(E))$, and\linebreak
$\dr_{\nabla^{\Hom}}\omega_{\tilde \nabla}=0$
\cite{Molino71a,ChStXu16,Jotz18e}.

The class
$\alpha_\nabla=[\omega_{\tilde\nabla}]\in H^1(B,\Hom(A/B,\End(E)))$ is
called the \textbf{Atiyah class of the representation of $B\subseteq A$
  on $E$}. It does not depend on the choice of the extension, and it is
zero if and only if there is an extension $\tilde\nabla$ such that
$R_{\tilde\nabla}(b,a)=0$ for all $b\in\Gamma(B)$ and all
$a\in\Gamma(A)$ \cite{Molino71a,ChStXu16,Jotz18e}. That is,
$\alpha_\nabla=0$ if and only if there is an extension $\tilde\nabla$
such that
$R_{\tilde\nabla}\in\Gamma(\wedge^2B^\circ\otimes\End(E))$. 

Then for all $l\geq 1$ the form
$\widehat{\tr}(R_{\tilde\nabla}^l)$ is a section of
$\wedge^{2l}B^\circ$ and so $\widehat{\tr}(R_{\tilde\nabla}^l)=0$ for
$2l>q$. This shows the following theorem.

\begin{theorem}\label{with_atiyah}
  Let $E$ be a smooth vector bundle over a smooth manifold $M$ and let
  $A$ be a Lie algebroid over $M$.  If there exists a Lie subalgebroid
  $B$ of $A$ of codimension $q$ with a linear representation
  $\nabla\colon \Gamma(B)\times\Gamma(E)\to\Gamma(E)$ with vanishing
  Atiyah class $\alpha_\nabla\in H^1(B,\Hom(A/B,\End(E)))$, then the
  Pontryagin spaces
  \[\Pont^l_A(E)\subseteq H^l(A)
  \]
  are all trivial for $l>q$.
\end{theorem}

If $A=TM$, $B=F_M$, and $\nabla$ is defined by a fibration to a vector
bundle over $M/F_M$ as in the previous section, then the Atiyah class
$\alpha_\nabla$ vanishes (see \cite{Jotz18e}).  With Section
\ref{reductible_vb}, this yields the following corollary.
\begin{corollary}
  Let $E$ be a smooth vector bundle over a smooth manifold $M$.  If
  there exists an involutive subbundle $F_M$ of $TM$ of codimension
  $q$ with a flat connection
  $\nabla\colon \Gamma(F_M)\times\Gamma(E)\to\Gamma(E)$ such that
\begin{equation*}
\begin{xy}
  \xymatrix{
    E\ar[d]_{q_E}\ar[r]^{\pi}&E/\nabla\ar[d]^{[q_{E}]}\\
    M\ar[r]_{\pi_M}&M/F_M }
\end{xy}
\end{equation*}
is a smooth fibration of vector bundles, then the Atiyah class\linebreak
$\alpha_\nabla\in H^1(F_M,\Hom(TM/F_M,\End(E)))$ vanishes and the
Pontryagin spaces\linebreak $\Pont^l(E)\subseteq H^l(M)$ are all trivial for
$l>q$.
  \end{corollary}

\section{Pontryagin algebras of graded
  vector bundles}\label{conn_hom_sec}
This section studies connections up to homotopy on graded vector
bundles, and explains  how Pontryagin or characteristic algebras are
defined by those objects, in the same manner as the classical
Pontryagin algebras of a vector bundle are defined by linear
connections on it \cite{Quillen85,Mehta14}.

\subsection{The graded trace operator}\label{grd_vb}

In the following, consider a Lie algebroid
$(A\to M, \rho, [\cdot\,,\cdot])$, and a graded vector bundle
$\E=\oplus_{z\in\mathbb Z}E_z[z]$ over the same smooth
manifold $M$, with grading concentrated in finitely many degrees
(i.e.~all but finitely many of the vector bundles $E_z$, $z\in\mathbb Z$ are
trivial).

The \textbf{graded trace operator}
  $\gtr\colon \Gamma(\underline{\End}(\E))\to C^\infty(M)$,
  i.e.~\[\gtr\in\Omega^0(A, \Hom(\underline{\End}(\E)_0,
    \mathbb R))\] is defined by \[\gtr(\phi)=(-1)^{i}\tr(\phi)\]
  for $\phi\in\Gamma(\End(E_i))$. It yields
  a (graded) $\Omega^\bullet(A)$-linear map
  \[\widehat{\operatorname{gtr}}\colon \Omega(A,\underline{\End}(\E))_\bullet\to \Omega^\bullet(A).
  \]
  The operator $\widehat{\gtr}$ vanishes by definition on
  $\Omega^\bullet(A,\underline{\End}(\E)_i)$ for all $i\neq 0$, and so
  only `sees' the part $\Omega^\bullet(A,\underline{\End}(\E)_0)$ of
  $\Omega(A,\underline{\End}(\E))_\bullet$.

  The signs are chosen such that for
  $K_1\in\Omega^0(A,\Hom(E_i,E_j))=\Gamma(\Hom(E_i,E_j))$ and
  $K_2\in\Omega^0(A,\Hom(E_j,E_i))=\Gamma(\Hom(E_j,E_i))$, i.e.~with
  compositions $K_1\circ K_2\in\Gamma(\End(E_j))$ and
  $K_2\circ K_1\in\Gamma(\End(E_i))$:
  \begin{equation*}
    \begin{split}
      \gtr(K_1\circ K_2)&=(-1)^j\tr(K_1\circ K_2)=(-1)^j\tr(K_2\circ K_1)\\
      &=(-1)^{i+j}\gtr(K_2\circ K_1)=\gtr((-1)^{(j-i)(i-j)}K_2\circ K_1)\\
      &= \gtr((-1)^{|K_1|\cdot|K_2|}K_2\circ K_1)
\end{split}
\end{equation*}
since $i+j$ and $(j-i)(i-j)=2ij-j^2-i^2$ have the same parity.
That is, 
\begin{equation}\label{gtr_sign_comm}
  \gtr[K_1,K_2]=0
\end{equation}
for $K_1\in\Gamma(\Hom(E_i,E_j))$ and $K_2\in\Gamma(\Hom(E_j,E_i))$. More generally, this yields  
    \begin{equation}\label{trace_commutator_0_eq}
      \widehat{\gtr}([K_1,K_2])=0
    \end{equation}
    for 
    $K_1,K_2\in\Omega(A, \underline{\End}(\E))_\bullet$, see also \cite{Quillen85}.

    


  \subsection{Connections up to homotopy}\label{conn_hom}


  The notion of superconnection dates back to Quillen
  \cite{Quillen85}. Connections up to homotopy appeared in
  \cite{GrMe10a} in the more recent literature. The notion of
  connection up to homotopy defined by Crainic in
  \cite{Crainic00a,Crainic00b} is a different\footnote{There, a
    connection up to homotopy on a $2$-term complex $(E,\partial)$ of
    vector bundles
    \[ E^0\underset{\partial}{\overset{\partial}{\leftrightarrows}}
      E^1
\] is an $\R$-bilinear map
$\nabla\colon \mx(M)\times\Gamma(E)\to\Gamma(E)$ such that
$\partial\circ\nabla=\nabla\circ\partial$, that satisfies as usual the
Leibniz condition in the second argument, but which is not
$C^\infty(M)$-linear in the $\mx(M)$-entry. Instead, the failure of
the $C^\infty(M)$-linearity is measured by the commutator of
$\partial$ with a map
$H_\nabla\colon C^\infty(M)\times\Gamma(E)\to\Gamma(\End(E))$, which
is $\R$-linear and local in its entries. } one. 
\medskip

  Let $A\to M$ be a Lie algebroid and let $\E\to M$ a
  graded vector bundle of finite rank, i.e.~the grading is
  concentrated in finitely many degrees. Then a \textbf{connection up to
    homotopy} of $A$ on $\E$ is an operator
  \[\mathcal D\colon \Omega(A,\E)_\bullet\to\Omega(A,\E)_{\bullet+1}
  \]
  that increases the total degree by $1$ and satisfies
  \begin{equation}\label{leibniz_n_conn}
    \mathcal D(\omega\wedge \eta)=\dr_A\omega\wedge \eta+ (-1)^{|\omega|}\omega\wedge\mathcal D\eta
    \end{equation}
  for $\omega\in\Omega^\bullet(A)$ and
  $\eta\in\Omega(A,\E)_{\bullet}$.

  \medskip If $\E=E_0[0]\oplus\ldots\oplus E_{n-1}[n-1]$, then a
  connection up to homotopy of $A$ on $\E$ is called for simplicity
  \textbf{an $n$-connection}.  Of course, an $n$-connection
  $\mathcal D$ is an \emph{$n$-representation}, i.e.~an $n$-term
  representation up to homotopy in the sense of \cite{ArCr12}, if in
  addition $\mathcal D^2=0$.

\begin{example}[Degree-preserving connections are connections up to homotopy]\label{usual_connections}
  Let $\E\to M$ be a graded vector bundle of finite
  rank. Choose for all $z\in\mathbb Z$ a linear $A$-connection
  $\nabla^z\colon\Gamma(A)\times\Gamma(E_z)\to\Gamma(E_z)$.  Then the
  connections define together a connection up to homotopy
  \[\mathcal D\colon \Omega(A,\E)_\bullet\to\Omega(A,\E)_{\bullet+1}
  \]
  by
  $ \mathcal D(\omega)=\dr_{\nabla^z}\omega\in \Omega^{\bullet+1}(A,E_z)$
  for $\omega\in \Omega^\bullet(A,E_z)$.
\end{example}

\begin{example}[$2$-connections in more detail]\label{2_conn_explicit}
  Take
  $\E=E_0[0]\oplus E_1[1]$ over $M$ and $A\to M$ a Lie
  algebroid. Then a $2$-connection \[\mathcal D\colon \Omega(A,\E)_\bullet\to\Omega(A,\E)_{\bullet+1}
  \]
  is completely defined by its values
  \[\mathcal D(e_0)\in\Omega^1(A,E_0)\oplus\Omega^0(A,E_1)\quad \text{
      and }\quad \mathcal
    D(e_1)\in\Omega^1(A,E_1)\oplus\Omega^2(A,E_0)\] for arbitrary
  $e_0\in\Gamma(E_0)$ and $e_1\in\Gamma(E_1)$. It is easy to check
  that \[\mathcal D(e_0)=\dr_{\nabla^0}e_0+\partial(e_0)
  \quad \text{ and } \quad \mathcal D(e_1)=\dr_{\nabla^1}e_1+\widehat{K}(e_1)
  \]
  for $\nabla^i\colon \Gamma(A)\times\Gamma(E_i)\to\Gamma(E_i)$ linear
  connections, $i=0,1$, a vector bundle morphism
  $\partial\colon E_0\to E_1$ over the identity,
  i.e.~$\partial\in\Omega^0(A,\Hom(E_0,E_1))$, and
  $K\in\Omega^2(A,\Hom(E_1,E_0))$.
\end{example}

In general, connections up to homotopy can be described as follows.
\begin{proposition}\label{dec_n_conn}
  Let $A\to M$ be a Lie algebroid and let $\E\to M$ be a
  graded vector bundle of finite rank. Then a connection up to homotopy
  \[\mathcal D\colon \Omega(A,\E)_\bullet\to\Omega(A,\E)_{\bullet+1}
  \]
  can always be written
  \[\mathcal D=\dr_\nabla+\widehat{D}
  \]
  with a linear connection
  $\nabla\colon\Gamma(A)\times\Gamma(\E)\to\Gamma(\E)
  $ that preserves the grading as in Example \ref{usual_connections}, and
  $D\in\Omega(A,\underline{\End}(\E))_1$.
  The connection $\nabla$ and the form $D$ can even be chosen such that
  \[D\in\bigoplus_{s\neq 1}\Omega^s(A,\underline{\End}(\E)_{1-s}).\]
\end{proposition}

\begin{proof}
  Take any degree-preserving connection
  $\nabla\colon\Gamma(A)\times\Gamma(\E)\to\Gamma(\E)$
  as in Example \ref{usual_connections}. Then
  $\mathcal D-\dr_\nabla\colon
  \Omega(A,\E)_\bullet\to\Omega(A,\E)_{\bullet+1}$
  is easily seen to be graded $\Omega^\bullet(A)$-linear. Hence
  $\mathcal D-\dr_\nabla=\widehat{D}$ for a
  $D\in\Omega(A,\underline{\End}(\E))_1$.

  \medskip

  Now write
  $D=\sum_{s\in\Z}D_s\in\bigoplus_{s\in
    \Z}\Omega^s(A,\underline{\End}(\E)_{1-s})$.  Then
  $D_1\in\Omega^1(A,\underline{\End}(\E)_0)$ and so
  $\nabla':=\nabla+D_1$ is a new connection on $\E$ that preserves
  the grading, such that
  $\mathcal D=\dr_\nabla+\widehat{D}=\dr_{\nabla'}+\widehat{D-D_1}$.
\end{proof}

\medskip

Finally, a connection up to homotopy of $A$ on $\E$ defines an induced
connection up to homotopy
  \[\mathcal D_{\End}\colon \Omega(A,\underline{\End}(\E))_\bullet\to \Omega(A,\underline{\End}(\E))_{\bullet+1}
    \]of $A$ on $\underline{\End}(\E)$ by
  \[ \widehat{\mathcal D_{\End}(K)}=\mathcal D\circ\widehat{K}-(-1)^{|K|}\widehat{K}\circ\mathcal D
  \]
  for all $K\in \Omega(A,\underline{\End}(\E))_\bullet$ and
  $e\in\Gamma(\E)$. That is, as before,
  \begin{equation}\label{def_D_End}
    [\mathcal D,\widehat{K}]:=\mathcal D\circ \widehat{K}- (-1)^{|K|}\widehat{K}\circ\mathcal D= \widehat{\mathcal D_{\End}K}
  \end{equation}
  for all $K\in \Omega(A,\underline{\End}(\E))_\bullet$.
  More generally, if $\mathcal D$ is a connection up to homotopy of $A$ on
  $\E$ and $\mathcal D'$ is a connection up to homotopy of $A$ on
  $\underline{E'}$, then define the induced connection up to
  homotopy
  \[\mathcal D_{\Hom}\colon
    \Omega(A,\underline{\Hom}(\E,\underline{E'}))_\bullet\to
    \Omega(A,\underline{\Hom}(\E,\underline{E'}))_{\bullet+1}
    \]of $A$ on $\underline{\Hom}(\E,\underline{E'})$ by
  \[ \widehat{\mathcal D_{\Hom}(K)}=\mathcal D'\circ \widehat{K}-(-1)^{|K|}\widehat{K}\circ\mathcal D
  \]
  for all $K\in \Omega(A,\underline{\Hom}(\E,\underline{E'}))_\bullet$.

  As in the case of superconnections, this yields the following lemma
  \cite{Quillen85}, see also \cite{Mehta14}.
  \begin{lemma}In the situation above, 
    \begin{equation}\label{trace_da}
      \widehat{\gtr}\circ \mathcal D_{\End}=\dr_A\circ \widehat{\gtr}.
    \end{equation}
  \end{lemma}
  
     \begin{proof}
      Write the connection up to homotopy $\mathcal D$ as in Proposition \ref{dec_n_conn} as
      \[ \mathcal D=\dr_\nabla+\widehat{D} \] with
      $\nabla\colon\Gamma(A)\times\Gamma(\E)\to\Gamma(\E)$ a linear
      connection that preserves the grading, and
      $D\in\Omega(A,\underline{\End}(\E))_1$.  Then for
      $K\in\Omega(A,\underline{\End}(\E))_\bullet$:
      \begin{equation*}
        \begin{split}
          \widehat{\mathcal D_{\End}K}
          &=\dr_\nabla\circ \widehat{K}+\widehat{D}\circ\widehat{K}-(-1)^{|K|}\widehat{K}\circ\dr_\nabla-(-1)^{|K|}\widehat{K}\circ\widehat{D}=\widehat{\dr_{\nabla^{\End}}K}+\widehat{[D,K]}.
        \end{split}
      \end{equation*}
      This yields $\mathcal D_{\End}K=\dr_{\nabla^{\End}}K+[D,K]$ and so by \eqref{trace_commutator_0_eq}
      \begin{equation}\label{gtr_after_D}
        \widehat{\gtr}\left(\mathcal D_{\End}K\right)=\widehat{\gtr}\left(\dr_{\nabla^{\End}}K\right).
      \end{equation}
      The connection $\mathcal \dr_{\nabla^{\End}}$ and the flat
      connection
      $\dr_A\colon \Omega^\bullet(A)\to\Omega^{\bullet+1}(A)$ yield as
      before the connection
      $\dr_{\nabla^{\Hom}}\colon
      \Omega(A,\underline{\Hom}(\underline{\End},\R))_\bullet\to
      \Omega(A,\underline{\Hom}(\underline{\End},\R))_{\bullet+1}$.
      Equation \eqref{gtr_after_D} and the proof of Lemma
      \ref{lemma_easy_trace_flat} now give
      \[\dr_A\circ\widehat{\gtr}-\widehat{\gtr}\circ\mathcal D_{\End}
        =\dr_A\circ\widehat{\gtr}-\widehat{\gtr}\circ\dr_{\nabla^{\End}}=\widehat{\dr_{\nabla^{\Hom}}\gtr}=0.\qedhere
      \]
    \end{proof}


\subsection{Curvature of a connection up to homotopy, and Pontryagin characters}
Now if $\mathcal D$ is a connection up to homotopy of $A$ on
$\E$, then \eqref{leibniz_n_conn} implies immediately
\[\mathcal D^2(\omega\wedge \eta)=\omega\wedge \mathcal D^2\eta=(-1)^{2|\omega|}\omega\wedge \mathcal D^2\eta
  \]
  for $\omega\in\Omega^\bullet(A)$ and
  $\eta\in\Omega(A,\E)_{\bullet}$. That is,
  $\mathcal D^{2}$ is (graded) $\Omega^\bullet(A)$-linear and there is
  a unique $R\in\Omega(A,\underline{\End}(\E))_2$ with
  $\mathcal D^2=\widehat{R_{\mathcal D}}$.
    The form $R_{\mathcal D}\in\Omega(A,\underline{\End}(\E))_2$ is the
  \textbf{curvature form of $\mathcal D$}.

    Of course, an $n$-connection is an $n$-representation if and only
    if its curvature form vanishes.  As before, define
    $R_{\mathcal
      D}^i\in\Omega(A,\underline{\End}(\E))_{2i}$ by
  \[ \widehat{R_{\mathcal D}^i}=\widehat{R_{\mathcal D}}^i=\mathcal D^{2i}
  \]
  for $i\geq 1$.
The Bianchi identity
  \begin{equation}\label{bianchi}
    \mathcal D_{\End}R_{\mathcal D}^i=0
  \end{equation}
then  holds for all $i\geq 1$ since
  \[ \widehat{\mathcal D_{\End}R^i_{\mathcal D}}=[\mathcal D,
    \widehat{R^i_{\mathcal D}}]=[\mathcal D, \mathcal D^{2i}]=\mathcal
    D^{2i+1}-(-1)^{2i}\mathcal D^{2i+1}=0.
  \]
  As a consequence, the curvature form $R_{\mathcal D}$ satisfies
    \[\dr_A(\widehat{\gtr}(R^{i}_{\mathcal D}))\overset{\eqref{trace_da}}{=}\widehat{\gtr}(\mathcal D_{\End}(R^{i}_{\mathcal D}))\overset{\eqref{bianchi}}{=}0
    \]
    for all $i\geq 1$.


   \begin{example}\label{usual_grd_curvature}
     In the situation of Example \ref{usual_connections}, it is easy
     to see that for each $i\geq 1$
     \[R_{\mathcal D}^i=\sum_{z\in \Z}R_{\nabla^z}^i
       \,\,\in\bigoplus_{z\in \Z}\Omega^{2i}(A,\End(E_z))\subseteq
       \Omega(A,\underline{\End}(\E))_{2i}.\] In this case, all the
     results follow easily from the considerations in
     \S\ref{classical_pont}, and
     \[\widehat{\gtr}(R^{i}_{\mathcal D})=\sum_{z\in \Z}(-1)^z\widehat{\tr}(R_{\nabla^z}^i)
     \]
     which is obviously a $\dr_A$-closed element of $\Omega^{2i}(A)$ by \eqref{closed_easy}.
     This is already observed in \cite{Quillen85} in the context of superconnections.
     \end{example}

    \bigskip

    Now one can construct as before the Pontryagin algebras defined by
    the powers of the curvature form.
    \begin{proposition}\label{inv_conn}
      Choose a graded vector bundle $\E$ of finite rank
      over a smooth manifold $M$, and a Lie algebroid $A$ over $M$.
      Then the cohomology classes
    \[\left[\widehat{\gtr}(R^{i}_{\mathcal D})\right]\in H^{2i}(A)
    \]
    do not depend on the choice of the connection up to homotopy
    $\mathcal D$ on $\E$.
  \end{proposition}
  As observed in \cite{Mehta14}, the proof of Proposition
  \ref{inv_conn} follows the standard techniques, eaxctly as done in
  \cite{Quillen85} in the situation of superconnections. For
  the convenience of the reader, it is carried out in detail in Appendix
  \ref{proof_inv_conn}.  
  \begin{definition}
    Choose a graded vector bundle $\E$ of finite rank over
    a smooth manifold $M$, and a Lie algebroid $A$ over $M$.  Then the
    \textbf{$A$-Pontryagin algebra of the graded vector bundle $\E$} 
    \[
      \Pont^\bullet_A(\E)\subseteq H^\bullet(A)
    \]
    is the subalgebra generated by the \textbf{$A$-Pontryagin characters of $\E$}
    \[\sigma_A^i(\E):=\left[\widehat{\gtr}(R^{i}_{\mathcal D})\right]\in H^{2i}(A), \quad i\geq 1,
    \]
    defined by any choice of connection up to homotopy $\mathcal D $ of $A$ on $\E$.


    \end{definition}

    Here also, it is easy to show using Example
    \ref{usual_connections} and Proposition \ref{inv_conn} that
    \[\Pont^\bullet_A(\E)=\rho^\star\Pont^\bullet(\E).
    \]
    As usual, $\Pont^\bullet(\E)$ denotes the $TM$-Pontryagin algebra
    of $\E$.

    \begin{remark}
      This paper does not define \emph{Pontryagin classes} of a graded
      vector bundle as images of special invariant polynomials under a
      suitable Chern-Weil homomorphism -- this is not needed for the
      obstruction theorems below. However, consider a graded vector
      bundle $\E=\bigoplus_{z\in\Z}E_z[z]$ and set
      $\underline{V}:=\bigoplus_{z\in\Z}\R^{\operatorname{rank}E_z}[z]$,
      a (finite dimensional) graded $\R$-vector space. Set
      $\mathcal A(\underline{V})\subseteq \mathcal
      P(\underline{\lie{gl}}(\underline{V}))$ to be the subalgebra of
      polynomials that is generated by the polynomials
      \[ \phi\mapsto \gtr(\phi^l),
      \]
      for $l\geq 1$.  Then there is an obvious Chern-Weil homomorphism
      $\mathcal A(\underline{V})\to \Pont^\bullet_A(\E)$ of $\R$-algebras, but
      $\mathcal A(\underline{V})$ cannot be understood as a subalgebra
      of the $\underline{\operatorname{Gl}}(\underline{V})$-invariant
      polynomials on $\underline{\lie{gl}}(\underline{V})$ since for
      $\phi\in \underline{\lie{gl}}(\underline{V})$,
      $A\in \underline{\operatorname{Gl}}(\underline{V})$ and $l\geq 1$:
      \[ \gtr((A\phi A\inv)^l)=(-1)^{|A|+l|\phi|\cdot|A|}\gtr(\phi^l)
      \]
      by \eqref{gtr_sign_comm}.
      \end{remark}

    \begin{example}
      In the situation of Example \ref{2_conn_explicit},
      \[\mathcal D\colon \Omega(A,E_0[0]\oplus E_1[1])_\bullet\to \Omega(A,E_0[0]\oplus E_1[1])_{\bullet+1}\]
      equals
      \[\mathcal D=\dr_\nabla+\widehat{\partial}+\widehat{\omega}
      \]
      with $\nabla\colon \Gamma(A)\times\Gamma(\E)\to\Gamma(\E)$ a
      linear connection that preserves the degree,
      $\partial\in\Gamma(\Hom(E_0,E_1))=\Omega^0(A,\End(\E)_1)$ and
      $\omega\in\Omega^2(A,\Hom(E_1,E_0))=\Omega^2(A,\End(\E)_{-1})$.

      Then
      \begin{equation}
        \begin{split}
          \mathcal D^2&=\dr_\nabla^2+\dr_\nabla\circ\widehat{\partial}+\dr_\nabla\circ\widehat{\omega}+\widehat{\partial}\circ\dr_\nabla+\widehat{\partial}\circ\widehat{\omega}+\widehat{\omega}\circ\dr_\nabla+\widehat{\omega}\circ\widehat{\partial}\\
          &=
          \dr_\nabla^2+\left[\dr_\nabla,\widehat{\partial}\right]+\left[\dr_\nabla,\widehat{\omega}\right]+\left[\widehat{\partial},\widehat{\omega}\right]
          =
          \widehat{R_\nabla}+\widehat{\dr_{\nabla^{\End}}\partial}+\widehat{\dr_{\nabla^{\End}}\omega}+\widehat{\left[\partial,\omega\right]}.
      \end{split}
    \end{equation}
    In this equation,
    $R_\nabla+\left[\partial,\omega\right]\in\Omega^2(A,\End(\E)_0)$,
    $\dr_{\nabla^{\End}}\partial\in\Omega^1(A,\End(\E)_1)$ and
    $\dr_{\nabla^{\End}}\omega\in\Omega^3(A,\End(\E)_{-1})$. This
    shows that the $2$-connection is a $2$-representation if and only
    if \cite{ArCr12,GrMe10a}
    \[ R_{\nabla^0}+\omega\circ\partial=0, \quad R_{\nabla^1}+\partial\circ\omega=0, \quad \nabla^1\circ \partial=\partial\circ\nabla^0\, \text{ and }\, \dr_{\nabla^{\End}}\omega=0.
    \]
    The form $\widehat{\gtr}(\mathcal D^2)$ is 
    \[ \widehat{\gtr}(R_\nabla+\left[\partial,\omega\right])=\widehat{\tr}(R_{\nabla^0}+\omega\circ\partial)-\widehat{\tr}(R_{\nabla^1}+\partial\circ\omega).
    \]
    The form $\widehat{\gtr}(\mathcal D^4)$ is the graded trace of 
    \[R_\nabla^2+R_\nabla\wedge[\partial,\omega]+[\partial,\omega]\wedge R_\nabla+[\partial,\omega]^2+\left(\dr_{\nabla^{\End}}\partial\right)\wedge\left(\dr_{\nabla^{\End}}\omega\right)
        +\left(\dr_{\nabla^{\End}}\omega\right)\wedge\left(\dr_{\nabla^{\End}}\partial\right),
      \]
      etc.
      \end{example}

\subsection{Application: Obstructions to the existence of an $n$-representation}
Example \ref{usual_connections} shows that a degree-preserving linear
$A$-connection on $\E$ is an example of an $A$-connection up to
homotopy on $\E$. Choose a graded vector bundle $\E$ of finite rank $k$ over a smooth
      manifold $M$, and a Lie algebroid $A$ over $M$, and set $E:=\oplus_{z\in\Z}E_z$.
      If $\E$ is concentrated in even degrees, then by Proposition \ref{inv_conn} and Example
\ref{usual_grd_curvature},
      the Pontryagin characters satisfy
    \[
      \sigma^i_A(\E)=\sigma^i(E)
    \]
    for all $i\geq 1$.
 If $\E$ has grading
    in odd degrees only,
    \[ \sigma_A^i(\E)=-\sigma^i_A(E)\in H^\bullet(A),
    \]
    for all $i\geq 1$.
    That is, the Pontryagin algebra of the graded vector bundle
    $\E$ is then just the Pontryagin algebra of the
    vector bundle $E$ obtained by forgetting the grading on
    $\E$.
    
    This shows that Pontryagin algebras of graded vector bundles only
    lead to new information if the grading is on mixed odd and even
    degrees. In general, Proposition \ref{inv_conn}, Example
    \ref{usual_connections} and Example \ref{usual_grd_curvature} lead
    to the following formula.

  \begin{corollary}\label{old_alt_sign_is_new}
    Let $\E=\bigoplus_{z\in\Z}E_z$ be a graded vector bundle of finite
    rank over a smooth manifold $M$, and let $A\to M$ be a Lie
    algebroid.
Then for  $l\geq 1$,
the $A$-Pontryagin character $\sigma_A^l(\E)$ of $\E$
equals
\begin{equation}\label{alt_form_pont}
\sigma_A^l(\E)=\sum_{z\in
    \Z}(-1)^z\sigma^l_A(E_z) \quad \in H^{2l}(A).
\end{equation}
\end{corollary}

\begin{proof}
Choose linear connections
$\nabla\colon \Gamma(A)\times\Gamma(E_z)\to\Gamma(E_z)$ for each
$z\in\Z$ and let
$\mathcal D$ be the induced connection up to homotopy of $A$ on $\E$
as in Example \ref{usual_connections}.
Then by Proposition \ref{inv_conn} and Example \ref{usual_grd_curvature}:
\[\sigma_A^l(\E)=\left[\widehat{\gtr}(R_{\mathcal D}^l)\right]=\sum_{z\in
   \Z}(-1)^z\left[\widehat{\tr}(R_{\nabla^z}^l)\right]=\sum_{z\in
    \Z}(-1)^z\sigma^l_A(E_z).\qedhere\]
  \end{proof}
\begin{remark}
  Using the formula in the last corollary, it is again easy to show
  that $\Pont_A^l(\E)\neq 0$ implies $l=4z$ for some $z\in \N$.
\end{remark}
  \medskip

   Corollary \ref{old_alt_sign_is_new} gives a
  necessary condition for the existence of an $n$-representation on a
  given graded vector bundle
  $\E=E_0[0]\oplus\ldots\oplus E_ {n-1}[n-1]$.
      \begin{theorem}
        Let $\E=E_0[0]\oplus\ldots\oplus E_{n-1}[n-1]$ be a graded
        vector bundle over a smooth manifold $M$, and let $A\to M$ be
        a Lie algebroid.  If there exists an $n$-representation
        $\mathcal D$ of $A$ on $\E$, then the $A$-Pontryagin
        characters $\sigma_A^l(E_i)$, $l>1$, of the vector bundles
        $E_i$, $i=0,\ldots,n-1$, satisfy the equations
        \begin{equation}\label{gen_n_rep_case}
          \sum_{i=0}^{n-1}(-1)^i\sigma^l_A(E_i)=0\in H^{2l}(A)
        \end{equation}
        for all $l>1$.
      \end{theorem}

      \begin{proof}
Since there is an $n$-connection $\mathcal D$ with $\mathcal D^2=0$, the left-hand side of \eqref{alt_form_pont} vanishes.
        \end{proof}

      \begin{theorem}\label{main_app}
        Let $E$ and $F$ be smooth vector bundles over $M$, and let
        $A\to M$ be a Lie algebroid. If there is a $2$-representation
        of $A$ on $E[0]\oplus F[1]$, then
        \[ \Pont^\bullet_A(E)=\Pont_A^\bullet(F) \subseteq H^\bullet(A).
        \]
        More precisely, the $A$-Pontryagin classes of $E$ equals the $A$-Pontryagin classes of $F$.
      \end{theorem}
      \begin{proof}
        In this case, \eqref{gen_n_rep_case} yields immediately
        \[\sigma^l_A(E)=\sigma^l_A(F)\]
        for all $l\geq 1$. Therefore, since the generators of the
        Pontryagin algebras are equal, the Pontryagin algebras and
        the Pontryagin classes of $E$ and $F$ must be equal.
      \end{proof}

      \bigskip
      
        The reader acquainted with the equivalence of decomposed
        VB-algebroids with $2$-representations \cite{GrMe10a}, and of
        decomposed double Lie algebroids with matched pairs of
        $2$-representations \cite{GrJoMaMe18} might find interesting
        the two following corollaries of Theorem \ref{main_app}.
        \begin{corollary}\label{main_2rep}
          Let $B$ and $C$ be smooth vector bundles over $M$, and let
          $(A\to M, \rho, [\cdot\,,\cdot])$ be a Lie algebroid. If
          there is a VB-algebroid $(D\to B, A\to M)$ with core $C$,
          then the total Pontryagin classes coincide:
          \[p_A(B)=p_A(C) \in H^\bullet(A).
          \]
          That is,
          $  \rho^\star p^l(B)=\rho^\star p^l(C)$ for all $l\geq 1$.
        \end{corollary}

        \begin{corollary}
          Let $C$ be a  smooth vector bundle over $M$, and let
          $A\to M$ and $B\to M$ be two Lie algebroids. If there is a double Lie algebroid
          $(D, A, B, M)$ with core $C$, then
          \[ p_A(C)=p_A(B) \in H^\bullet(A),
        \quad \text{ and } \quad  p_B(C)=p_B(A) \in H^\bullet(B).
          \]
        \end{corollary}

      \subsubsection{Example: the double $2$-representation defined by a connection}\label{sec_double_rep}
      Let $A\to M$ be a Lie algebroid and $E$ a vector bundle over
      $M$. Then any linear $A$-connection
      $\nabla\colon\Gamma(A)\times\Gamma(E)\to\Gamma(E)$ defines as
      follows a representation up to homotopy of $A$ on
      $E[0]\oplus E[1]$, see \cite{ArCr12,GrMe10a}. The operator
      \[
        \mathcal D\colon \Omega(A,E[0]\oplus E[1])_\bullet\to \Omega(A,E[0]\oplus E[1])_{\bullet+1}
      \]
      is defined by
      \[ \mathcal D(e_0)=\dr_\nabla(e_0)+e_0\in\Omega^1(A,E[0])\oplus \Omega^0(A,E[1])
      \]
      for $e_0\in\Omega^0(A,E[0])=\Gamma(E[0])$, and
      \[ \mathcal D(e_1)=\dr_\nabla(e_1)-\widehat{R_\nabla}(e_1)\in\Omega^1(A,E[1])\oplus \Omega^2(A,E[0])
      \]
      for $e_1\in\Omega^0(A,E[1])=\Gamma(E[1])$. Here, $R_\nabla$ is
      seen as an element of\linebreak
      $\Omega^2(A,\Hom(E[1],E[0]))\subseteq
      \Omega(A,\underline{\End}(\E))_1$.  It is easy to
      check that $\mathcal D^2=0$.
      This representation up to homotopy is called the \emph{double
        representation up to homotopy} of $A$ on $E$ \cite{ArCr12,GrMe10a}.


          Let $K\subseteq E$ be a vector subbundle. Take an
          $A$-connection $\nabla^K$ on $K$ and an $A$-connection
          $\bar\nabla$ on $E/K$. Then $K\oplus E/K\simeq E$ and the
          sum $\nabla^K+\bar\nabla$ defines an $A$-connection on
          $E$. The $A$-Pontryagin characters of $E$, $K$ and $E/K$ satisfy
        \[ \sigma_A^i(E)=\sigma_A^i(K)+\sigma_A^i(E/K) \in H^{2i}(A)
        \]
        for $i\geq 0$. This is usually formulated as
        $p_A(E)=p_A(K)\wedge p_A(E/K)$ (see e.g.~\cite{MiSt74,Tu17}).
        In other words the generators of $\Pont^\bullet_A(E/K)$ are
        given by
        \begin{equation}\label{gen_class_EK}
          \sigma_A^i(E/K)=\sigma_A^i(E)-\sigma_A^i(K) \in H^{2i}(A)
        \end{equation}
        for $i\geq 0$.  Likewise, the linear $A$-connections on $K$
        and on $E$ define together a $2$-connection $\mathcal D$ of
        $A$ on $K[0]\oplus E[1]$. Hence, the $A$-Pontryagin characters
        of $K[0]\oplus E[1]$ are
        \begin{equation}\label{gen_class_EK}
         \widehat{\gtr}(R^i_{\mathcal D})=\sigma_A^i(K)-\sigma_A^i(E) \in H^{2i}(A)
        \end{equation}
        for $i\geq 0$. Up to a sign, they equal the generators of $\Pont^\bullet_A(E/K)$. This yields the following proposition.
        \begin{proposition}\label{prop_eq_of_ponts}
          Let $E\to M$ be a smooth vector bundle, and let $A$ be a Lie algebroid over $M$. Let $K\subseteq E$ be a vector subbundle of $E$.
          Then
        \begin{equation}\label{eq_of_ponts}
          \Pont^\bullet_A(K[0]\oplus E[1])=\Pont^\bullet_A(E/K).
          \end{equation}
     \end{proposition}

\subsubsection{Example: the adjoint $2$-representation of a Lie algebroid}\label{sec_adjoint}
Let $A\to M$ be a Lie algebroid with anchor $\rho$ and Lie bracket
$[\cdot\,,\cdot]$. Then any choice of linear connection
$\nabla\colon\mx(M)\times\Gamma(A)\to\Gamma(A)$ defines as follows a
representation up to homotopy of $A$ on $A[0]\oplus TM[1]$, see
\cite{ArCr12,GrMe10a}. The operator
\[
        \mathcal D_{\ad}\colon \Omega(A,A[0]\oplus TM[1])_\bullet\to \Omega(A,A[0]\oplus TM[1])_{\bullet+1}
      \]
      is defined by
      \[ \mathcal D_{\ad}(a)=\dr_{\nabla^{\rm bas}}(a)+\rho(a)\in\Omega^1(A,A[0])\oplus \Omega^0(A,TM[1])
      \]
      for $a\in\Omega^0(A,A[0])=\Gamma(A)$, and
      \[ \mathcal D_{\ad}(X)=\dr_{\nabla^{\rm bas}}(X)-\widehat{R^{\rm  bas}_\nabla}(X)\in\Omega^1(A,TM[1])\oplus \Omega^2(A,A[0])
      \]
      for $X\in\Omega^0(A,TM[1])=\mx(M)$. Here, $R^{\rm bas}_\nabla\in\Omega^2(A,\Hom(X,A))$ is defined by
      \[R^{\rm bas}_\nabla(a,b)X=-\nabla_X[a,b]+[\nabla_Xa,b]+[a,\nabla_Xb]+\nabla_{\nabla_b^{\rm bas}X}a-\nabla_{\nabla_a^{\rm bas}X}b        \]
        for $a,b\in\Gamma(A)$ and $X\in\mx(M)$, and the two basic
        connections
      \[\nabla^{\rm bas}\colon\Gamma(A)\times\mx(M)\to\mx(M)  \quad\text{ and } \quad \nabla^{\rm bas}\colon\Gamma(A)\times\Gamma(A)\to\Gamma(A)
      \]
      are defined by
      \[ \nabla^{\rm bas}_aX=[\rho(a),X]+\rho(\nabla_Xa), \qquad
        \nabla^{\rm bas}_ab=[a,b]+\nabla_{\rho(b)}a
      \]
      for $a,b\in\Gamma(A)$ and $X\in\mx(M)$.  A computation shows
      $R_{\mathcal D_{\ad}}=\mathcal D_{\ad}^2=0$.

      This representation up to homotopy is called the
      \emph{adjoint representation up to homotopy} of $A$ on $E$
      \cite{ArCr12,GrMe10a}.  The following result follows from Theorem \ref{main_app}
     \begin{theorem}
  Let $A$ be a vector bundle over a smooth manifold $M$, and let
  $\rho\colon A\to TM$ be a vector bundle morphism over the identity.
  If $A\to M$ carries a Lie algebroid structure with anchor $\rho$, then
        \[ \rho^\star \left(p^l(A)\right)=\rho^\star \left(p^l(TM) \right)\in H^{4l}(A)
        \]
        for all $l\geq 1$.
      \end{theorem}


      \subsubsection{Example: the  $2$-representations defined by a morphism of Lie algebroids}\label{sec_adjoint}
      More generally, let $A\to M$ and $B\to M$ be two Lie algebroids,
      with a Lie algebroid morphism $\partial\colon B\to A$ over the
      identity on $M$.  Then any choice of linear connection
      $\nabla\colon\Gamma(A)\times\Gamma(B)\to\Gamma(B)$ defines as
      follows a representation up to homotopy of $B$ on
      $B[0]\oplus A[1]$ -- this was found in the work in preparation
      \cite{JoMa14}.

      The operator
\[\mathcal D\colon \Omega(B,B[0]\oplus A[1])_\bullet\to \Omega(B,B[0]\oplus A[1])_{\bullet+1}
      \]
      is defined by
      \[ \mathcal D(b)=\dr_{\nabla^{\partial}}(b)+\partial(b)\in\Omega^1(B,B[0])\oplus \Omega^0(B,A[1])
      \]
      for $b\in\Omega^0(B,B[0])=\Gamma(B)$, and
      \[ \mathcal D(X)=\dr_{\nabla^{\partial}}(a)-\widehat{R^{\partial}_\nabla}(a)\in\Omega^1(B,A[1])\oplus \Omega^2(B,B[0])
      \]
      for $a\in\Omega^0(B,A[1])=\Gamma(A)$. Here, $R^{\partial}_\nabla\in\Omega^2(A,\Hom(X,A))$ is defined by
      \[R^{\partial}_\nabla(b_1,b_2)a=-\nabla_a[b_1,b_2]+[\nabla_ab_1,b_2]+[b_1,\nabla_ab_2]+\nabla_{\nabla_{b_2}^{\partial}a}b_1-\nabla_{\nabla_{b_1}^{\partial}a}b_2        \]
        for $b_1,b_2\in\Gamma(B)$ and $a\in\Gamma(A)$, and the two
        connections
      \[\nabla^{\partial}\colon\Gamma(B)\times\Gamma(B)\to\Gamma(B)  \quad\text{ and } \quad \nabla^{\partial}\colon\Gamma(B)\times\Gamma(A)\to\Gamma(A)
      \]
      are defined by
      \[ \nabla^{\partial}_{b_1}b_2=[b_1,b_2]+\nabla_{\partial b_2}b_1, \qquad
        \nabla^{\partial}_ba=[\partial(b),a]+\partial(\nabla_{a}b)
      \]
      for $b_1,b_2\in\Gamma(B)$ and $a\in\Gamma(A)$.  A computation
      shows $\mathcal D^2=0$ and so $R_{\mathcal D}=0$.  The following
      result follows then from Theorem \ref{main_app}.
      \begin{theorem}
        Let $A$ and $B$ be Lie algebroids over $M$. If there is a Lie
        algebroid morphism $\partial\colon B\to A$ over the identity
        on $M$, then
        \[p_B^l(A)=p_B^l(B)
        \]
        for all $l\geq 1$.
      \end{theorem}

      Vaisman defines characteristic classes of morphisms of Lie
      algebroids in \cite{Vaisman10}; by considering the graphs of
      these morphisms. The result above does not consider these
      classes; but it would be interesting to compare the two
      approaches.
      
        \subsection{Bott's vanishing theorem for  graded vector bundles}
        This section proves a more general formulation of Bott's
        vanishing theorem \cite{Bott72} and of Theorem \ref{main}, on
        Lie subalgebroids with $n$-representations. 

        For $B\subseteq A$ a subalgebroid, the space
        $\Omega(B,\E)_\bullet$ can be (non-canonically) embedded as
        follows as a $C^\infty(M)$-submodule of
        $\Omega(A,\E)_\bullet$. Fix $C\subseteq A$ a subbundle such
        that $A=B\oplus C$.  Then the $C^\infty(M)$-linear map
        $i_C\colon\Omega(B,\E)_\bullet\to\Omega(A,\E)_\bullet$ is
        defined by
        \[i_C(\omega)(a_1,\ldots,a_s)=\omega(b_1,\ldots,b_s)
          \]
          for $\omega\in \Omega^s(B,E_i)$ and
          $a_j=b_j+c_j\in A=B\oplus C$ for $j=1,\ldots,s$.  In the
          same manner, 
          $i_C\colon\Omega(B,\underline{\End}(\E))_\bullet\to\Omega(A,\underline{\End}(\E))_\bullet$ is defined.

          In addition, the inclusion $\iota\colon B\to A$ induces the $C^\infty(M)$-linear restriction map
          \[ \iota^\star\colon \Omega(A,\E)_\bullet\to \Omega(B,\E)_\bullet,
          \]
          defined by
          $(\iota^\star\omega)(b_1,\ldots,b_s)=\omega(b_1,\ldots,b_s)$
          for $\omega\in\Omega^s(A,E_l)$. By construction,
          $\iota^\star\circ
          i_C=\Id_{\Omega(B,\underline{\End}(\E))_\bullet}$.

          \bigskip

        Let now $A\to M$ be a Lie algebroid and
        $\E=E_0[0]\oplus\ldots\oplus E_{n-1}[n-1]$ be a graded vector
        bundle over $M$. Let $k$ be the rank of $\E$.  Assume that
        there is a Lie subalgebroid $B\subseteq A$ of codimension $q$,
        with an $n$-representation
        \[ \mathcal D\colon \Omega(B,\E)_\bullet\to\Omega(B,\E)_{\bullet+1}.
        \]
        Then, as in Proposition \ref{dec_n_conn}, the
        $n$-representation $\mathcal D$ equals
        $\mathcal D=\dr_\nabla+\widehat{D}$, with a $B$-connection on
        $\E$ preserving the grading, and a form
        $D\in\Omega(B,\underline{\End}(\E))_1$. Using the second part
        of Proposition \ref{dec_n_conn}, without loss of generality
        $D$ has no component in
        $\Omega^1(B,\underline{\End}(\E)_0)$. Extend the
        $B$-connection $\nabla$ on $\E$ to an $A$-connection
        $\tilde \nabla$ on $\E$ that preserves the grading, and extend
        the form $D$ to the form
        $i_C(D)\in\Omega(A,\underline{\End}(\E))_1$, after the choice
        of a smooth complement $C$ of $B$ in $A$.

        Then
        \[ \widetilde{\mathcal D}=\dr_{\tilde\nabla}+\widehat{i_C(D)}\colon
          \Omega(A,\E)_\bullet\to\Omega(A,\E)_{\bullet+1}
        \]
        is an $n$-connection of $A$ on $\E$.
        Take $\omega\in \Omega^s(A,E_l)$. Then
        \[\widetilde{\mathcal D}(\omega)=\sum_{i=0}^{\rank
            A}(\widetilde{\mathcal D}\omega)_i\in
          \bigoplus_{i=0}^{\rank A}\Omega^i(A,E_{s+l+1-i})
        \]
        and easy computations yield the following identities:
        \begin{itemize}
        \item For $i=s+1$: \begin{equation*}
            \begin{split}
              (\widetilde{\mathcal D}\omega)_i(b_1,\ldots,b_i)&=(\dr_{\tilde\nabla}\omega)(b_1,\ldots,b_i)=\dr_\nabla(\iota^\star\omega)(b_1,\ldots,b_i)\\
              &=(\mathcal D(\iota^\star\omega))_i(b_1,\ldots,b_i)
            \end{split}
            \end{equation*}
         for $b_1,\ldots,b_i\in\Gamma(B)$, and 
        \item For $i\neq s+1$:
          \begin{equation*}
            \begin{split}
              (\widetilde{\mathcal D}\omega)_i(b_1,\ldots,b_i)&=(\widehat{i_C(D)}\omega)_i(b_1,\ldots,b_i)=(\widehat{D}(\iota^{\star}\omega))_i(b_1,\ldots,b_i)\\
              &=(\mathcal D(\iota^\star\omega))_i(b_1,\ldots,b_i)
              \end{split}
            \end{equation*}
            for $b_1,\ldots,b_i\in\Gamma(B)$.
          \end{itemize}
        
          This proves
          $\iota^\star\circ \widetilde{\mathcal D}=\mathcal D\circ
          \iota^\star$ and as a consequence $\iota^\star\circ \widetilde{\mathcal D}^2=\mathcal D^2\circ
          \iota^\star$. Therefore, the equality $\mathcal D^2=0$ yields
          $\iota^\star(R_{\tilde{\mathcal D}})=\iota^\star(\widetilde{\mathcal D}^2)=0$.
          That is, \[R_{\tilde{\mathcal D}}\in
          \left(I^\bullet(B)\otimes_{C^\infty(M)}\Gamma(\End(\E))\right)_2=\bigoplus_{j\geq 1}  I^j(B)\otimes_{C^\infty(M)}\Gamma(\End(\E)_{2-j}),\]
        and so
        \[ R_{\tilde{\mathcal D}}^l\in (I^\bullet(B))^l\otimes_{C^\infty(M)}\Gamma(\End(\E))
        \] for all $l\geq 1$.  This yields
        $R_{\tilde{\mathcal D}}^l=0$ for $l>q$, and, as in the
        classical case, the following theorem.

        \begin{theorem}\label{Bott_gen_graded}
           Let $A\to M$ be a Lie algebroid and let
        $\E=E_0[0]\oplus\ldots\oplus E_{n-1}[n-1]$ be a
        graded vector bundle over $M$. Assume that there is a Lie
        subalgebroid $B\subseteq A$ of codimension $q$, with an $n$-representation
        \[ \mathcal D\colon \Omega(B,\E)_\bullet\to\Omega(B,\E)_{\bullet+1}.
        \]
        Then the $A$-Pontryagin spaces of the graded vector bundle $\E$
        \[ \Pont_A^l(\E)\subseteq H^l(A)
        \]
        all vanish for $l>2q$.
      \end{theorem}

      \begin{example}
        Let $E\to M$ be a smooth vector bundle, and let $A$ be a Lie
        algebroid over $M$. Let $K\subseteq E$ be a vector subbundle
        of $E$ and let $B\subseteq A$ be a subalgebroid.  Consider a
        linear $A$-connection $\nabla$ on $E$, that preserves $K$.
        Define the linear $B$-connection
        $\bar\nabla\colon\Gamma(B)\times\Gamma(E/K)\to\Gamma(E/K)$ by
        $\bar\nabla_b\overline{e}=\overline{\nabla_be}$
        for all $b\in\Gamma(B)$ and $e\in\Gamma(E)$, where
        $\overline{e}\in\Gamma(E/K)$ is the class of the section $e$.

        The connection $\bar\nabla$ is flat if and only if the
        $2$-representation of $A$ on $E[0]\oplus E[1]$ defined by
        $\nabla$ as in \S\ref{sec_double_rep} restricts to a
        $2$-representation of $B$ on $K[0]\oplus E[1]$; see
        \cite{DrJoOr15}. Then, by Theorem \ref{Bott_gen_graded},
        \[ \Pont_A^l(K[0]\oplus E[1])\subseteq H^l(A)
        \]
        all vanish for $l>2q$. By \eqref{eq_of_ponts}, this
        is  a reformulation of Theorem \ref{main} in the graded setting.
     \end{example}

\section{Infinitesimal ideal systems and Pontryagin classes}\label{iis_sec}
The main motivation for the results above was the search for
obstructions to the existence of infinitesimal ideal systems in a
given Lie algebroid, in terms of the $A$ and $TM$-Pontryagin classes
of $A$ and $TM$. This section first recalls some of the main examples
of infinitesimal ideal systems. Then the first and second subsections
present the obtained obstructions.  \medskip

Recall that infinitesimal ideal systems are defined as in the  Definition on Page \pageref{iis_def}.
The three main classes of examples of infinitesimal ideal systems are
the following. 
\begin{example}[The usual notion of ideals in Lie algebroids]\label{naive_ideal}
  An ideal $I$ in a Lie algebroid $A\to M$ is a subbundle over $M$
  such that $[a,i]\in\Gamma(I)$ for all $i\in\Gamma(I)$ and all
  $a\in\Gamma(A)$.  The inclusion $I\subseteq\ker(\rho)$ follows
  immediately and shows that this definition of an ideal is very
  restrictive. These ideals, called here \emph{naive ideals},
  correspond obviously to the ideal systems $(F_M=0,J=I,\nabla=0)$ in
  $A$.  In particular, an ideal in a Lie algebra is an infinitesimal
  ideal system.
\end{example}

\begin{example}[The Bott connection]\label{foliation}
  Consider an involutive subbundle $F_M\subseteq TM$ and the Bott
  connection
\[\nabla^{F_M}\colon\Gamma(F_M)\times\Gamma(TM/F_M)\to\Gamma(TM/F_M)\]
associated to it. Then it is straightforward to check that the triple
$( F_M, F_M, \nabla^{F_M})$ is an infinitesimal ideal system in the
Lie algebroid $TM$.
\end{example}

\begin{example}[The ideal system associated to a fibration of Lie algebroids]\label{fibration}
Let \begin{equation*}
\begin{xy}
\xymatrix{
A\ar[r]^{\varphi}\ar[d]_{q_A}& A'\ar[d]^{q_{A'}}\\
M\ar[r]_f&M'
}
\end{xy}
\end{equation*}
be a fibration of Lie algebroids, i.e.~the map
$\varphi_0$ is a surjective submersion and $\varphi^!\colon A\to \varphi_0^!A'$ 
is a surjective vector bundle morphism over the identity on $A$.

Then $J:=\ker(\varphi)\subseteq A$ is a subalgebroid of $A$ and
$F_M=T^{\varphi_0}M\subseteq TM$ is an involutive subbundle.  The equality
$T\varphi_0\circ\rho=\rho'\circ\varphi$ yields immediately $\rho(J)\subseteq
F_M$.

Define a connection $\nabla^\varphi\colon \Gamma(F_M)\times
\Gamma(A/J)\to\Gamma(A/J)$
by setting $\nabla^\varphi_X\bar a=0$ for 
all sections $a\in \Gamma(A)$ that are $\varphi$-related to some
section $a'\in\Gamma(A')$, i.e.~such that $\varphi\circ a=a'\circ \varphi_0$.
Then the properties of the Lie algebroid morphism $(\varphi, \varphi_0)$ imply 
that $(F_M, J, \nabla^\varphi)$ is an infinitesimal ideal system in $A$.

Conversely, the following theorem shows that, up to topological
obstructions, a Lie algebroid can be ``quotiented out'' by an
infinitesimal ideal system \cite{JoOr14}, just as a Lie algebra modulo
an ideal gives a new Lie algebra.  More precisely let $(F_M,J,\nabla)$
be an infinitesimal ideal system in a Lie algebroid $A$. Assume that
$\bar M=M/F_M$ is a smooth manifold and that $\nabla$ has trivial
holonomy.  Then the quotient defined by parallel transport along the
leaves of $F_M$, $(A/J)/\nabla$, inherits a Lie algebroid structure
over $F/F_M$ such that the canonical projections
$\pi\colon A\to (A/J)/\nabla$ and $\pi_M\colon M\to M/F_M$ define a
fibration of Lie algebroids \cite{JoOr14}.

\end{example}

\subsection{Pontryagin classes associated to an infinitesimal ideal
  system}

First of all, since an infinitesimal ideal system consists among other
ingredients of an involutive subbundle $F_M\subseteq TM$ and a flat
$F_M$-connection on $A/J$,  the following proposition is immediate.
\begin{proposition}\label{prop_easy_ob1}
  Let $(F_M,J,\nabla)$ be an infinitesimal ideal system in a Lie
  algebroid $A\to M$. Let $q$ be the codimension of $F_M$ in $TM$.
  Then the Pontryagin algebras $\Pont^r(A/J)$ and $\Pont^r(TM/F_M)$
  are all trivial for $r>2q$.
\end{proposition}

Next, it is easy to see that $J$ is  a subalgebroid of $A$.
The Bott connection associated to $J\subseteq A$ is the flat
$J$-connection $\nabla^J$ on $A$ defined by
\[ \nabla^J\colon\Gamma(J)\times\Gamma(A/J)\to\Gamma(A/J),\quad \nabla^J_j\overline{a}=\overline{[j,a]}
\]
for $j\in\Gamma(J)$ and $a\in\Gamma(A)$. In addition, there is a flat
$J$-connection on $TM/F_M$, defined by
\[\nabla\colon\Gamma(J)\times\Gamma(TM/F_M)\to\Gamma(TM/F_M),\quad \nabla_j\overline{X}=\overline{[\rho(j),X]}
\]
for $j\in\Gamma(J)$ and $X\in\mx(M)$.
This, Proposition \ref{prop_easy_ob1} and Remark \ref{imp_rem} yield the following result.
\begin{proposition}\label{prop_easy_ob2}
  Let $(F_M,J,\nabla)$ be an infinitesimal ideal system in a Lie
  algebroid $A\to M$. Let $s$ be the codimension of $J$ in $A$.  Then
  the Pontryagin algebras $\Pont^r_A(A/J)$ and $\Pont^r_A(TM/F_M)$ are all trivial
  for $r>2\min\{s,q\}$.
\end{proposition}

Of course, Propositions \ref{prop_easy_ob1} and \ref{prop_easy_ob2}
can be refined using Theorem \ref{with_atiyah} and the Atiyah classes
defined by extensions of the four flat connections.

\subsection{Finer obstructions}

The obstructions found above are too ``rough'' for being really
meaningful -- the proofs use very little of the structure of
infinitesimal ideal systems.  This section uses the Pontryagin
algebras of graded vector bundles in order to find further (finer!)
obstructions to the existence of infinitesimal ideal systems in a
given Lie algebroid.

In order to do this, let us recall some results found in
\cite{DrJoOr15}.  Let $A\to M$ be a Lie algebroid. Let
$F_M\subseteq TM$ be an involutive subbundle and let $J\subseteq A$ be
a smooth subbundle. Let
$\nabla\colon \Gamma(F_M)\times\Gamma(A/J)\to\Gamma(A/J)$ be a flat
connection, and let
$\tilde\nabla\colon\mx(M)\times\Gamma(A)\to\Gamma(A)$ be an extension
of $\nabla$. That is, $\tilde\nabla_Xj\in\Gamma(J)$ for all
$X\in\Gamma(F_M)$ and $j\in\Gamma(J)$ and the induced quotient
connection equals $\nabla$.  Recall from \S\ref{sec_adjoint} that
$\tilde\nabla$ defines the two basic connections
  \[ \tilde\nabla^{\rm bas}\colon\Gamma(A)\times\Gamma(A)\to
    \Gamma(A), \qquad \tilde\nabla^{\rm
      bas}\colon\Gamma(A)\times\mx(M)\to \mx(M)
  \]
  and the basic curvature
  $R_{\tilde\nabla}^{\rm bas}\in\Omega^2(A,\Hom(TM,A))$ -- that is,
  $\tilde\nabla$ defines the adjoint representation
  $\ad_{\tilde\nabla}$ as in \S\ref{sec_adjoint}.
  
  Then $(F_M,J,\nabla)$ is an infinitesimal ideal system in $A$
  if and only if \cite{DrJoOr15}:
  \begin{enumerate}
    \item $\rho(J)\subseteq F_M$;
  \item The basic connection 
    $\tilde\nabla^{\rm bas}\colon
    \Gamma(A)\times\Gamma(A)\to\Gamma(A)$ preserves $J$;
  \item The basic connection
    $\tilde\nabla^{\rm bas}\colon \Gamma(A)\times\mx(M)\to\mx(M)$
    preserves $F_M$;
  \item The basic curvature
    $R_{\tilde\nabla}^{\rm bas}\in\Omega^2(A,\operatorname{Hom}(TM,A))$ restricts to
    an element of \linebreak$\Omega^2(A,\operatorname{Hom}(F_M,J))$.
  \end{enumerate}
  That is, the adjoint $2$-representation
  $\ad_{\tilde\nabla}$ of $A$ on $A[0]\oplus TM[1]$ defined by the
  anchor and the basic connections and curvature restricts to a
  $2$-representation of $A$ on $J[0]\oplus F_M[1]$. Theorem
  \ref{main_2rep} yields immediately the following result.
\begin{theorem}\label{adv_ob_1}
  Let $(A\to M, \rho, [\cdot\,,\cdot])$ be a Lie algebroid.  Let
  $J\subseteq A$ and $F_M\subseteq TM$ be vector subbundles. If $F_M$
  is involutive and there is a flat $F_M$-connection on $A/J$ such
  that $(F_M,J,\nabla)$ is an infinitesimal ideal system, then
  \[ p^l_A(J)=p^l_A(F_M)\in H^{4l}(A)
  \]
  for all $l\geq 1$.
  \end{theorem}
  \begin{example}\label{3_ex_1}
    Example \ref{naive_ideal} and the last proposition show that if $I\subseteq A$ is an ideal, then
    $\Pont^\bullet_A(I)=\{0\}$.
    This is easy to see directly since $A$ is represented on $I$ by the Lie bracket.
    \medskip

    In the situation of Example \ref{foliation}, the statement of the
    last proposition is trivial since $J=F_M$.
However, Example \ref{fibration} and the last proposition show
    that if $\varphi\colon A\to A'$ 
is a fibration of Lie algebroids over a smooth submersion $f\colon M\to M'$, then
    \[ p^l_A(T^fM)=p^l_A(\ker\varphi)\in H^{4l}(A)
    \]
    for all $l\geq 1$.
  \end{example}

\appendix
\section{Proof of Proposition \ref{inv_conn}}\label{proof_inv_conn}

Let $\mathcal D$ and $\mathcal D'$ be two connections up to homotopy
of the Lie algebroid $A$ on a graded vector bundle
$\E=\oplus_{k\in\Z}E_z[z]$ of finite rank.
The difference $\mathcal D'-\mathcal D$ is
graded-$\Omega^\bullet(A)$-linear and there exists an element
$D\in\Omega(A,\underline{\End}(\E))_1$ such that
$ \mathcal D'-\mathcal D=\widehat{D}$.
For each $t\in[0,1]$ set $\mathcal D_t=\mathcal
D+t\widehat{D}$. Then $\mathcal D_t$ is a connection up to homotopy of
$A$ on $\E$ for all $t\in[0,1]$, with
$\mathcal D_0=\mathcal D$ and $\mathcal D_1= \mathcal D'$.
Its curvature at time $t$ reads
$ \widehat{R_{\mathcal D_t}}= \mathcal D_t^2=(\mathcal D+t\widehat{D})^2=\mathcal D^2+t\left[\mathcal D,\widehat{D}\right]+\frac{1}{2}t^2\left[\widehat D, \widehat D\right]$,
which leads to 
\begin{equation*}
  \frac{d}{dt}\widehat{R_{\mathcal D_t}}= \left[\mathcal D,\widehat{D}\right]+t\left[\widehat D, \widehat D\right]=\left[\mathcal D_t,\widehat{D}\right]
  \overset{\eqref{def_D_End}}{=}\widehat{\mathcal D_{t,\End}D},
\end{equation*}
and so to $\frac{d}{dt} R_{\mathcal D_t}=\mathcal D_{t,\End}D$.
Next, this implies
\begin{equation*}
  \frac{d}{dt} R_{\mathcal D_t}^i=\sum_{s=1}^iR_{\mathcal D_t}^{(s-1)}\wedge \mathcal D_{t,\End}D\wedge R_{\mathcal D_t}^{i-s}
\end{equation*}
and so
\begin{equation*}
  \begin{split}
    \frac{d}{dt}\widehat{\gtr}\left(R_{\mathcal
        D_t}^i\right)
    &\,\,=i\cdot \widehat{\gtr}\left(R_{\mathcal D_t}^{i-1}\wedge \mathcal D_{t,\End}D\right)
    \overset{\eqref{bianchi}}{=}i\cdot \widehat{\gtr}\left(\mathcal D_{t,\End}\left(R_{\mathcal D_t}^{i-1}\wedge D\right)\right)\\
    &\overset{\eqref{trace_da}}{=}i\cdot\dr_A\left(\widehat{\gtr}\left(R_{\mathcal
          D_t}^{i-1}\wedge D\right)\right).
  \end{split}
\end{equation*}
Using this, conclude that
\[ \widehat{\gtr}\left(R_{\mathcal
        D}^i\right)- \widehat{\gtr}\left(R_{\mathcal
        D'}^i\right)=\dr_A\int_0^1i\cdot\left(\widehat{\gtr}\left(R_{\mathcal
          D_t}^{i-1}\wedge D\right)\right)dt,
  \]
 and so $\widehat{\gtr}\left(R_{\mathcal
        D}^i\right)$ and $\widehat{\gtr}\left(R_{\mathcal
        D'}^i\right)$ define the same cohomology class in $H^{2i}(A)$.

\def\cprime{$'$} \def\polhk#1{\setbox0=\hbox{#1}{\ooalign{\hidewidth
  \lower1.5ex\hbox{`}\hidewidth\crcr\unhbox0}}} \def\cprime{$'$}
  \def\cprime{$'$}

\end{document}